\documentclass[12pt]{article}

\usepackage{amsmath,amssymb,latexsym, amsfonts, amscd, amsthm, mathrsfs, xy,verbatim}

\usepackage{pdfsync}
\usepackage[french]{babel}
\usepackage{hyperref}
\xyoption{all}

\theoremstyle{plain}

\newtheorem{ithm}{Theorem}
\newtheorem{iconj}{Conjecture}
\newtheorem{theorem}{Theorem}[section]

\newtheorem{prop}[theorem]{Proposition}

\newtheorem{lemma}[theorem]{Lemma}

\theoremstyle{definition}
\newtheorem{definition}[theorem]{Definition}
\newtheorem{remark}[theorem]{Remark}

\long\def\symbolfootnote[#1]#2{\begingroup
\def\thefootnote{\fnsymbol{footnote}}\footnote[#1]{#2}\endgroup}

\def\lra{\longrightarrow}
\def\SL{{\bf SL}}
\def\GL{{\bf GL}}
\def\PGL{{\bf PGL}}

\def\cA{\mathcal{A}}
\def\sgn{\mathrm{sgn}}
\def\N{\mathrm{N}}
\def\1{{\mathbf 1}}

\DeclareMathOperator{\coeff}{coeff}
\DeclareMathOperator{\diag}{diag}

\DeclareMathOperator{\Hom}{Hom} 

\DeclareMathOperator{\sign}{sign}

 \DeclareMathOperator{\Tr}{Tr}
\DeclareMathOperator{\ord}{ord}

 \DeclareMathOperator{\real}{Re}

\DeclareMathOperator{\res}{res}
\DeclareMathOperator{\cores}{cores}
\DeclareMathOperator{\hd}{hd}
\DeclareMathOperator{\Irr}{Irr}

\DeclareMathOperator{\Image}{Image}
\DeclareMathOperator{\id}{id}
\DeclareMathOperator{\weight}{weight}

\newcommand{\Sh}{\mathrm{Sh}}

\newcommand{\stack}[2]{\genfrac{}{}{0pt}{}{#1}{#2}}

\newcommand{\mscr}{\mathscr }

\def\l{\ell}
\def\fa{\mathfrak{a}}

\def\fp{\mathfrak{p}}

\def\fZ{\mathfrak{Z}}

\def\cN{\mathcal{N}}

\def\Z{\mathbf{Z}}

\def\Q{\mathbf{Q}}
\def\C{\mathbf{C}}

\def\B{\mathbf{B}}
\def\D{\mathbf{D}}
\def\R{\mathbf{R}}

\def\bdf{\begin{defn}}
\def\edf{\end{defn}}
\def\cM{\mathcal{M}}
\def\cK{\mathcal{K}}
\def\cL{\mathcal{L}}

\def\cD{\mathcal{D}}
\def\cO{\mathcal{O}}

\def\cR{\mathcal{R}}
\def\cP{\mathcal{P}}
\def\cJ{\mathcal{J}}
\def\cV{\mathcal{V}}

\def\ff{\mathfrak{f}}
\def\fb{\mathfrak{b}}
\def\fc{\mathfrak{c}}

\def\Gal{{\rm Gal}}

\def\ra{\rightarrow}

\def\sP{\mscr{P}}

\def\mc{\mathcal}

\def\cF{{\cal F}}
\def\cQ{{\cal Q}}

\begin{document}
\title{Integral Eisenstein cocycles on $\GL_n$, II: \\  Shintani's method }

\author{Pierre Charollois\footnote{Partially supported by the grant  R\'EGULATEURS ``ANR-12-BS01-0002''} \\ Samit Dasgupta\footnote{Partially supported by NSF CAREER grant DMS-0952251}
 \\ Matthew Greenberg\footnote{Partially supported by an NSERC discovery grant}}

\maketitle

\begin{abstract}
We define a cocycle on $\GL_n(\Q)$ using Shintani's method.  
This construction is closely related to earlier work of Solomon and Hill, but
differs in that the cocycle property is achieved through the introduction of an auxiliary perturbation vector $Q$.
As a corollary of our result we obtain a new proof of a theorem of Diaz y Diaz and Friedman on signed fundamental
domains, and give a cohomological reformulation of Shintani's proof of the Klingen--Siegel rationality theorem
on partial zeta functions of totally real fields.

Next we  relate the Shintani cocycle to the Sczech cocycle by showing that the two differ by the sum of an explicit coboundary and  a simple ``polar'' cocycle.  This generalizes
a result of Sczech and Solomon in the case $n=2$. 

Finally, we introduce an integral version of our cocycle by smoothing at an auxiliary prime $\ell$.
This integral refinement has strong arithmetic consequences. We showed in previous work 
 that certain specializations of the smoothed class yield the $p$-adic $L$-functions
of totally real fields.  Furthermore, combining our cohomological construction with a theorem of Spiess, one deduces that  that the order of vanishing of these $p$-adic $L$-functions is at least as large as the
expected one.

\end{abstract}

 \tableofcontents

\section*{Introduction}

In this paper, we study a certain ``Eisenstein cocycle" on $\GL_n(\Q)$ defined using Shintani's method.  Our construction follows previous works of Solomon, Hu,  Hill, Spiess, and Steele in this direction (\cite{So}, \cite{husolomon}, \cite{hill}, \cite{spiess}, \cite{steele}).

We study three main themes in this paper.  First, we define an $(n-1)$-cocycle on $\GL_n(\Q)$ valued in a certain space of power series denoted $\R((z))^{\hd}$.  The basic idea of defining a  cocycle using Shintani's method is well-known; the value of the cocycle on a tuple of matrices is the Shintani--Solomon generating series associated to the simplicial cone whose generators are the images of a fixed vector under the action of these matrices.  The difficulty in defining a cocycle stems from two issues: choosing which boundary faces to include in the definition of the cone, and dealing with degenerate situations when the generators of the cone do not lie in general position.  Hill's method is to embed $\R^n$ into a certain ordered field with $n$ indeterminates, and to perturb the generators of the cone using these indeterminates so that the resulting vectors are always in general position.  The papers \cite{steele} and \cite{spiess} use Hill's method.
Our method is related, but somewhat different.  We choose an auxiliary irrational vector $Q \in \R^n$ and include a face of the simplicial cone if perturbing the face by this vector brings it into the interior of the cone.   
We learned during the writing of this paper
that this perturbation idea was studied much earlier by Colmez in unpublished work 
for the purpose of constructing Shintani domains \cite{colmez}.  Colmez's technique was  used by Diaz y Diaz and Friedman in \cite{ddf}.
However the application of this method to the cocycle property appears to be novel.

 Using formulas of Shintani and Solomon, we prove that the cocycle we construct specializes under the cap product with certain
 homology classes to yield the special values of partial zeta functions of totally real fields of degree $n$ at nonpositive integers.
This is a cohomological reformulation of Shintani's calculation of these special values and his resulting proof of the Klingen--Siegel theorem
on their rationality.

In 1993, Sczech introduced in \cite{sczech} an  Eisenstein  cocycle  on $\GL_n(\Q)$ that enabled him to give  another proof of the Klingen--Siegel theorem.  Our second main result is that the cocycles defined using  Shintani's method and Sczech's method are in fact cohomologous.  The fact that such a result should hold has long been suspected by experts in the field;  all previous attempts were restricted to the case $n=2$ (see for instance  \cite{Scznote}, \cite[\S7]{SoJNT} or \cite[\S5]{hill}). 
One technicality is that the cocycles are naturally defined with values in different modules, so  we first  define a common module where the cocycles can be compared, and then we provide an explicit coboundary relating them.

The  third  and final theme explored in this paper is 
 a smoothing process that allows for the definition of an {\em integral} version of the Shintani cocycle.  The smoothing method was introduced in  our earlier paper \cite{pcsd}, where we defined an integral version of the Eisenstein cocycle constructed by Sczech.
 The integrality property of the smoothed cocycles has strong arithmetic consequences.
We showed in \cite{pcsd} that one can use the smoothed Sczech--Eisenstein cocycle to construct the $p$-adic $L$-functions of totally real fields and furthermore to study the analytic behavior of these $p$-adic $L$-functions at $s=0$.  In particular, we showed using work of Spiess \cite{sphmf}
 that the order of vanishing of these $p$-adic $L$-functions at $s=0$ is at least equal to the expected one, as conjectured by Gross in \cite{gross}.  The formal nature of our proofs implies that these arithmetic results could be deduced entirely from the integral version of the Shintani cocycle constructed in this paper. In future work, we will explore further the leading terms of these $p$-adic $L$-functions at $s=0$ using our cohomological method  \cite{ds}.

We conclude the introduction by stating our results in greater detail and indicating the direction of the proofs.
Sections~\ref{s:compare} and~\ref{s:smooth} both rely on Sections~\ref{s:cocycle} and~\ref{s:zeta} but are independent from each other. Only Section~\ref{s:smooth} uses results from the earlier paper   [CD].

\subsection*{$Q$-perturbation, cocycle condition and fundamental domains}

Fix an integer $n \ge 2$, and let $\Gamma = \GL_n(\Q)$.  Let $\cK$ denote the abelian group of
functions on $\R^n$ generated by the characteristic functions of rational open simplicial cones, i.e.\ 
sets of the form $\R_{>0} v_1 + \R_{>0}v_2 + \cdots + \R_{>0}v_r$ with linearly independent $v_i \in \Q^n$.

Let $\R^n_{\Irr} \subset \R^n$ denote the set of
vectors with the property that their $n$ components are linearly independent over $\Q$.
Let $\cQ$ denote the set of equivalence classes of $\R^n_{\Irr}$ under multiplication by $\R_{>0}$.

Given an $n$-tuple of matrices $A = (A_1, \dotsc, A_n) \in \Gamma^n$, we let 
$\sigma_i \in \Q^n$ denote the leftmost column of $A_i$, i.e.\ the image under $A_i$
of the first standard basis vector.  (In fact replacing this basis vector by any nonzero vector in $\Q^n$ would suffice.)
Fixing $Q \in \R^n_{\Irr}$, we define an element $\Phi_{\mathrm{Sh}}(A, Q) \in \cK$ as follows.
If the $\sigma_i$ are linearly dependent, we simply let $\Phi_{\mathrm{Sh}}(A, Q) = 0$. 
If the $\sigma_i$ are linearly independent,
we define $\Phi_{\mathrm{Sh}}(A, Q) \in \cK$  to be the characteristic function of
the simplicial cone $C = C(\sigma_1, \dotsc, \sigma_n)$ and some of its boundary faces, multiplied by $\sgn(\det(\sigma_1, \dotsc, \sigma_n))$.
A boundary face is included if translation of an element of that face by a small positive multiple of $Q$ moves the element into the interior of $C$.
The property $Q \in \R^n_{\Irr}$ ensures that $Q$ does not lie in  any face of the cone, and hence  translation by a small multiple of $Q$ moves any element of a face into either the interior or exterior of the cone. 
The definition of $\Phi_{\mathrm{Sh}}(A, Q)$ depends on $Q$ only up to its image in $\cQ$.

 Our first key result is the following
cocycle property of $\Phi_{\mathrm{Sh}}$ (see Theorems~\ref{t:cqcocycle} and~\ref{t:phis}).  The function
\begin{equation} \label{e:phiscoc}
 \sum_{i=0}^{n} (-1)^i  \Phi_{\mathrm{Sh}}(A_0, \dotsc, \hat{A}_i, \dotsc, A_n, Q) 
 \end{equation}
lies in the subgroup $\cL \subset \cK$ generated by characteristic functions of wedges, i.e.\ sets
of the form $\R v_1 + \R_{>0}v_2 + \cdots + \R_{>0}v_r$ for some $r \ge 1$ and linearly independent $v_i \in \Q^n$. 
We conclude that the function $\Phi_{\mathrm{Sh}}$ 
defines
a homogeneous $(n-1)$-cocycle on $\Gamma$ valued in the space $\cN$ of functions $\mc Q \rightarrow \cK/\cL$.

Along the way we note that if the $\sigma_i$ are all in the positive orthant of $\R^n$, then in fact
the function (\ref{e:phiscoc}) vanishes.  As a result we obtain another proof of the main theorem of \cite{ddf},
which gives an explicit signed fundamental domain for the action of the group of totally positive units
in a totally real field of degree $n$ on the positive orthant.
In the language of \cite{spiess}, we show that the specialization of $\Phi_{\mathrm{Sh}}$ to the  unit group
 is a {\em Shintani cocycle} (see Theorem~\ref{t:ddf} below).

Using this result and Shintani's explicit formulas for the special values of  zeta functions 
associated to simplicial cones, we recover the following classical result originally proved by Klingen and Siegel.
Let $F$ be a totally real field, and let $\fa$ and $\ff$ be relatively prime integral ideals of $F$.
The partial zeta function of $F$ associated to the narrow ray class of $\fa$ modulo $\ff$ is defined by
\begin{equation} \label{e:zetadef}
 \zeta_{\ff}(\fa, s) = \sum_{\fb \sim_\ff \fa} \frac{1}{\N\fb^s}, \qquad \text{Re}(s) > 1. 
 \end{equation}
Here the sum ranges over integral ideals $\fb \subset F$ equivalent to $\fa$ in the narrow ray class group modulo $\ff$, which we denote $G_\ff$.
The function $\zeta_\ff(\fa, s)$ has a meromorphic continuation to $\C$, with only a simple pole at $s=1$.

\begin{ithm} \label{t:ks} The values $\zeta_{\ff}(\fa, -k)$ for integers $k \ge 0$ are rational.
\end{ithm}

We prove Theorem~\ref{t:ks} by showing that
\begin{equation} \label{e:zetapairing}
\zeta_{\ff}(\fa, -k) = \langle \Phi_{\mathrm{Sh}}, \fZ_k \rangle
\end{equation} where $\fZ_k \in H_{n-1}(\Gamma, \cN^\vee)$
is a certain homology class depending on $\fa, \ff$, and $k$, and the indicated pairing is the cap product
\begin{equation} \label{e:nv}
 H^{n-1}(\Gamma, \cN) \times H_{n-1}(\Gamma, \cN^\vee) \lra \R, \qquad \cN^\vee = \Hom(\cN, \R). 
 \end{equation}
See Theorem~\ref{t:psiszeta} below for a precise statement.
Combined with a rationality property of our cocycle (Theorem~\ref{t:rational}) that implies that the
cap product $\langle \Phi_{\mathrm{Sh}}, \fZ_k \rangle$ lies in $\Q$, we deduce the desired result.

Our proof of Theorem~\ref{t:ks} is simply a cohomological reformulation of Shintani's original argument.
However, our construction has the
benefit that we  give an explicit signed fundamental domain.  This latter feature is useful for computations
and served as a motivation for \cite{ddf} as well.

\subsection*{Comparison with the Sczech cocycle}

Sczech's  proof of Theorem~\ref{t:ks}   is deduced from  an identity similar to (\ref{e:zetapairing}), but involving a different cocycle.
It leads to explicit formulas  in terms of Bernoulli numbers  that resemble  those of Shintani in \cite{Sh}. A natural question that emerges is whether a direct comparison of the 
two constructions is possible. Our next result, stated precisely in Theorem~\ref{t:compare}, is a proof that the  cocycle on $\Gamma$ defined in Sections~\ref{s:cocycle} and~\ref{s:zeta} using Shintani's method is
cohomologous (after projecting to the $+1$-eigenspace for the action of $\{\pm 1\}$ on $\cQ$) 
to the cocycle defined by Sczech,  up to a simple and minor error term.  Rather than describing the details of Sczech's construction in this
introduction, we content ourselves with explaining the combinatorial mechanism enabling the proof, with an informal discussion
in the language of \cite[\S 2.2]{sczcom}.

For $n$ vectors $\tau_1, \dotsc, \tau_n \in \C^n$, define a rational function of a variable $x \in \C^n$ by
\[ f(\tau_1, \dotsc, \tau_n)(x) = \frac{\det(\tau_1, \dotsc, \tau_n)}{ \langle x, \tau_1 \rangle \cdots \langle x, \tau_n \rangle}.
\]
Given an $n$-tuple of matrices $A = (A_1, \dotsc, A_n) \in \Gamma^n$,  denote by
$A_{ij}$ the $j$th column of the matrix $A_i$.  The function $f$ satisfies a cocycle property (see (\ref{e:fcoc})) that implies that
the assignment $A\mapsto \alpha(A) := f(A_{11}, A_{21}, \dotsc, A_{n1})$ defines a homogeneous $(n-1)$-cocycle on $\Gamma$ valued in
the space of functions on Zariski open subsets of $\C^n$.  The rational function $\alpha(A)$ is not defined on the hyperplanes $\langle x, A_{i1} \rangle = 0$.

Alternatively we consider, for each $x \in \C^n - \{0 \}$, the index $w_i = w_i(A, x)$ giving the leftmost column of $A_i$ not orthogonal
to $x$.  The function $\beta(A)(x) = f(A_{1w_1}, \dotsc, A_{nw_n})(x)$ is then defined on $ \C^n - \{0 \}$, and the assignment $A \mapsto \beta(A)$ 
can also be viewed as a homogeneous $(n-1)$-cocycle on $\Gamma$.

Using an explicit computation, we show that the function $\alpha$ corresponds to our Shintani cocycle (Proposition~\ref{p:shinalpha}), 
whereas the function $\beta$ yields Sczech's cocycle (Proposition~\ref{p:sczbeta}).  A  coboundary relating $\alpha$ and $\beta$ is then given as follows.
Let $A = (A_1, \dotsc, A_{n-1}) \in \Gamma^{n-1}$, and define for $i=1, \dotsc, n-1$:
\[ h_i(A) =  
\begin{cases} f(A_{1w_1}, \dotsc, A_{(i-1)w_{i-1}}, A_{i1}, A_{iw_i}, A_{(i+1)1}, \dotsc, A_{(n-1)1}) & \text{ if } w_i > 1 \\
0 & \text{ if } w_i = 1.
\end{cases}
\]
Let $h = \sum_{i=1}^{n-1} (-1)^i h_i.$   We show that $\beta - \alpha = dh$.  In the case $n=2$, this recovers
Sczech's formula \cite[Page 371]{sczcom}.

\subsection*{Smoothing and applications to classical and $p$-adic $L$-functions}

In Section~\ref{s:smooth}, we fix a prime $\ell$ and we introduce a smoothed version $\Phi_{\mathrm{Sh},\ell}$ of the Shintani cocycle,
essentially by taking a difference between $\Phi_{\mathrm{Sh}}$ and a version of the same shifted by a matrix of
determinant $\ell$.  
The smoothed cocycle is defined on an arithmetic subgroup  $\Gamma_\ell \subset \Gamma$
and shown to satisfy an integrality property (Theorem~\ref{t:smooth}).

Through the connection
of the Shintani cocycle to zeta values given by (\ref{e:zetapairing}),  this integrality property translates  as in [CD]
into corresponding results about special values of zeta functions.  For the interest of the reader, we
have included the statements of these arithmetic results in this introduction.  For the proofs we refer the reader to \cite{pcsd}, where these applications
were already presented.

Our first arithmetic application of the smoothed cocycle is the
following integral refinement of
Theorem~\ref{t:ks}, originally due to Pi.~Cassou-Nogu\`es \cite{cn} and Deligne--Ribet \cite{dr}.
\begin{ithm} \label{t:integral} 
 Let $\fc$ be an integral ideal of $F$ relatively prime to $\ff$ and let $\ell = \N\fc$. The smoothed zeta function
\[  \zeta_{\ff, \fc}(\fa, s) = \zeta_{\ff}(\fa\fc, s) - \N\fc^{1-s}\zeta_\ff(\fa, s) \]
assumes values in $\Z[1/\ell]$ at nonpositive integers $s$.
\end{ithm}

Cassou--Nogu\`es' proof of   Theorem~\ref{t:integral} is a  refinement of Shintani's method
under the assumption that   $\cO_F/\fc$ is cyclic. The proof of Theorem~\ref{t:integral} that follows from the constructions in this paper
  is essentially a cohomological reformulation of Cassou--Nogu\`es' argument.
For simplicity we assume further  that $\ell = \N\fc$ is prime.   One can define a modified version 
of the homology class $\fZ_k$ denoted $\fZ_{k, \ell}$, such that $ \zeta_{\ff, \fc}(\fa, -k) = \langle\Phi_{\mathrm{Sh},\ell}, \fZ_{k, \ell} \rangle$.
A result from \cite{pcsd} restated in Theorem~\ref{t:smooth} below implies that the cap product
 $\langle \Phi_{\mathrm{Sh},\ell}, \fZ_{k, \ell} \rangle$
lies in $\Z[1/\ell]$, thereby completing the proof of Theorem~\ref{t:integral}.

The final arithmetic application of our results regards the study of the $p$-adic $L$-functions associated to abelian characters of the totally real field $F$.
Let $ \psi\colon \Gal(\overline{F}/F) \lra \overline{\Q}^* $ be a totally even finite order character.  Fix embeddings 
$\overline{\Q} \hookrightarrow \C$ and $\overline{\Q} \hookrightarrow \overline{\Q}_p$, so that $\psi$ can be viewed as taking values in $\C$ or $\overline{\Q}_p$. 
   Let $ \omega\colon \Gal(\overline{F}/F) 
  \longrightarrow \mu_{p-1} \subset \overline{\Q}^* $ denote\symbolfootnote[2]{As usual, replace $\mu_{p-1}$ by $\{\pm 1\}$ when $p=2$.}  the Teichm\"uller character.  Using the integrality properties of our cocycle
  $\Phi_{\mathrm{Sh},\ell}$, one  recovers the following theorem of  Cassou-Nogu\`es \cite{cn}, Barsky
  \cite{barskypapier}  and  Deligne--Ribet \cite{dr}. 

 \begin{ithm} \label{t:padic}
 There is a unique meromorphic $p$-adic $L$-function  $L_{p}(\psi, s)\colon \Z_p \lra \C_p$ satisfying the interpolation property
\[  L_{p}(\psi, 1-k) = L^*(\psi \omega^{-k}, 1-k) \]
 for integers $k \ge 1$, where $L^*$ denotes the  classical $L$-function with Euler factors at the primes dividing $p$ removed.
The function $L_p$ is analytic if $\psi \neq 1$.  If $\psi = 1$, there is 
at most a simple pole at $s=1$ and no other poles.
\end{ithm}

Now consider the totally odd character $\chi=\psi\omega^{-1}$, and let $r_\chi$ denote the number of primes $\fp$ of $F$ above $p$ such that $\chi(\fp) = 1$.  
In \cite{gross}, Gross proposed the following:
\begin{iconj}[Gross] \label{c:gross} We have
 \[ \ord_{s=0} L_p(\psi, s) = r_\chi. \]
\end{iconj}
Combining our cohomological construction of the $p$-adic $L$-function with Spiess's formalism, one obtains the following partial result towards Gross's conjecture:
\begin{ithm} \label{t:oov}
We have
 \[ \ord_{s=0} L_{p}(\psi, s) \ge r_\chi. \]
\end{ithm}
In the case $p > 2$, the result of Theorem~\ref{t:oov} was already known from Wiles' proof of the Iwasawa Main Conjecture \cite{Wi}.
Our method contrasts with that of Wiles in that it is purely analytic; we calculate the $k$th derivative of $L_{\fc, p}(\chi\omega, s)$ at $s=0$  and show that it equals the cap product 
of a cohomology class derived from $\Phi_{\mathrm{Sh},\ell}$ with a certain homology class denoted $\fZ_{\log^k}.$
 Spiess' theorem that the classes $\fZ_{\log^k}$ vanish
for $k < r_\chi$ then concludes the proof.
Our method  applies equally well when $p=2$.

Spiess proved Theorem~\ref{t:oov} as well using his formalism and his alternate construction of a Shintani cocycle \cite{spiess}.  
Note that our cocycle $\Phi_{\mathrm{Sh}}$ is ``universal" in the sense that it is defined on the group $\Gamma=\GL_n(\Q)$, whereas the cocycles defined by
Spiess are restricted to subgroups arising from unit groups in totally real number fields. (See Section~\ref{rs:specialize} below, where we describe
how our universal cocycle $\Phi_{\mathrm{Sh}}$ can be specialized to yield cocycles defined on unit groups.)

We should stress that while our proofs of Theorems~\ref{t:ks}, \ref{t:integral} and \ref{t:padic} are merely cohomological reformulations
of the works of Shintani \cite{Sh} and Cassou--Nogu\`es \cite{cn}, the proof of Theorem~\ref{t:oov}
relies essentially on the present cohomological construction and Spiess' theorems on cohomological $p$-adic $L$-functions.
In upcoming work we explore further the application of the cohomological method towards the leading terms of  $p$-adic
$L$-functions at $s=0$ and their relationship to Gross--Stark units \cite{ds}.

\bigskip

It is a pleasure to thank Pierre Colmez, Michael Spiess, and Glenn Stevens for helpful discussions and to acknowledge the influence of 
their papers \cite{colmezresidue}, \cite{spiess}, and \cite{st} on this work.  The first author  thanks 
 Alin Bostan and  Bruno Salvy  for many related discussions that stressed the importance  of power series methods. In March 2011, the second and third authors  gave a course at the Arizona Winter School that discussed Eisenstein cocycles.  The question of proving that the Shintani and Sczech cocycles are cohomologous was considered by the students in our group:  Jonathan Cass, Francesc Castella, Joel Dodge, Veronica Ertl, Brandon Levin, Rachel Newton, Ari Shnidman, and Ying Zhang.  A complete proof was given for the smoothed cocycles in the case $n=2$.  We would like to thank these students and the University of Arizona for an exciting week in which some of the ideas present in this work were fostered.

\section{The Shintani cocycle} \label{s:cocycle}

\subsection{Colmez perturbation} \label{s:qperturb}

Consider linearly independent vectors $v_1, \dotsc, v_n \in \R^m$.  The open cone generated by the $v_i$ 
is the set \[ C(v_1, \dotsc, v_n) = \R_{> 0} v_1 +
\R_{>0} v_2 + \cdots + \R_{>0} v_n. \]
We denote the characteristic function of this open cone by ${\mathbf 1}_{C(v_1, \dotsc, v_n)}$.  By convention, when $n=0$, we define
$C(\emptyset) = \{0 \}$.
Let $\cK_\R$ denote the abelian group of functions $\R^m \rightarrow \Z$ generated by the characteristic functions of such open cones. 

Fix now a subspace $V \subset \R^m$ spanned by arbitrary vectors $v_1, \dotsc, v_n \in \R^m$, and an auxiliary vector $Q \in \R^m$.
We define a function $c_Q(v_1, \dotsc, v_n) \in \cK_\R$ as follows.
If the $v_i$ are linearly dependent, then $c_Q(v_1, \dotsc, v_n) = 0$. 
If the $v_i$ are linearly independent, we impose the further condition that $Q \in V$ but that $Q$ is not in the $\R$-linear span of any subset of $n-1$ of the $v_i$. 
The function $c_Q(v_1, \dotsc, v_n)$ is defined to be the characteristic function of $C_Q(v_1, \dotsc, v_n)$, which is the disjoint union of the
 open cone $C(v_1, \dotsc, v_n)$
 and some of its boundary faces (of all dimensions, including 0).  A boundary face of the open cone $C$ is included in $C_Q$
 if translation of an element of the face by a small positive  multiple of $Q$ sends that element into the interior of $C$.  
  Formally,  we have:
\begin{equation} \label{e:cqdef}
 c_Q(v_1, \dotsc, v_n)(w) = \begin{cases}
\lim_{\epsilon \rightarrow 0^+}  {\mathbf 1}_{C(v_1, \dotsc, v_n)}(w + \epsilon Q) & \text{if  the } v_i \text{ are linearly independent,} \\
 0 & \text{otherwise.} \end{cases}
 \end{equation}
The limit in (\ref{e:cqdef}) is easily seen to exist and is given explicitly as follows.  If $w \not\in V$, then $c_Q(v_1, \dotsc, v_n)(w) = 0$.
On the other hand if
\[ w = \sum_{i=1}^{n} w_i v_i, \qquad Q = \sum_{i=1}^{n} q_i v_i  \ \ \text{ (all } q_i \neq 0), \]
then 
\begin{equation} \label{e:cqinterp}
 c_Q(v_1, \dotsc, v_n)(w) = \begin{cases} 1 & \text{if } w_i \ge 0 \text{ and } w_i = 0 \Rightarrow q_i > 0 \text{ for } i =1, \dotsc, n, \\
0 & \text{otherwise}.
\end{cases}
\end{equation}

Let us give one more characterization of this ``$Q$-perturbation process" that will be useful for future calculations.
For simplicity we suppose $m=n$ and that the vectors $v_i$ are linearly independent.
We denote by $\sigma$ the $n \times n$ matrix whose columns are the vectors $v_i.$
For each subset $I \subset \{1, \dotsc, n\}$, we have the open cone $C_I = C(v_i : i \in I)$.
The weight of this cone (equal to 0 or 1) in the disjoint union $C_Q$
is given as follows.
Let $d = |I|$.
 The $d$-dimensional subspace
 containing the cone $C_I$ can be expressed  as the intersection of the $n-d$ codimension 1 hyperplanes
 determined by $v_i^* = 0$, for $i \in \overline{I} = \{1, \dotsc, n\} - I$.  Here $\{v_i^*\}$ is the dual basis to the $v_i$. Under the usual
 inner product on $\R^n$, the $v_i^*$ are  the columns of the matrix $\sigma^{-t}$. 
 Each hyperplane $v_i^*=0$ divides its complement into a plus part and minus part, namely the half-space containing the cone $C(v_1, \dotsc, v_n)$ 
and the half-space not containing the cone (as an inequality, $\langle w, v_i^*\rangle > 0$ or $<0$).
The weight of $C_I$ is equal to 1 if $Q$ lies in the totally positive region defined by these hyperplanes, i.e. if
$\langle Q, v_i^*\rangle > 0$ for all $i\in \overline{I}.$  Otherwise, the weight of $C_I$ is $0$.
In summary,  \begin{equation} \label{e:weight}
\text{weight}(C_I) =  \prod_{i \in \overline{I}} \frac{1+ \sign(Q\sigma^{-t})_i }{2}. \end{equation}
  Note that this formula is valid for $d=n$ as well, with the standard convention that empty products are equal to 1.

\subsection{Cocycle relation}

We now derive a cocycle relation satisfied by the functions $c_Q$.  Let $v_1, \dotsc, v_n \in \R^m$ be linearly
independent vectors, with $n \ge 1$.
A set of the form
\begin{equation} \label{e:ldef}
 L = \R v_1 + \R_{>0} v_2 + \cdots + \R_{>0} v_n\end{equation}
is called a  {\em wedge}.  
The characteristic function ${\mathbf 1}_L$ of $L$ is an element of  $\cK_\R$ since 
\[ L = C(v_1, \dotsc, v_n) \sqcup C(v_2, \dotsc, v_n) \sqcup
C(-v_1, v_2, \dotsc, v_n). \]
Let  $\cL_\R = \cL_\R(\R^m) \subset \cK_\R$ be the subgroup
generated by the functions ${\mathbf 1}_L$ for all wedges $L$.

\begin{theorem}  \label{t:cqcocycle}
Let $n \ge 1$, and let $v_0, \dotsc, v_n \in \R^m$ be nonzero vectors spanning a subspace $V$ of dimension at most $n$.
 Let $Q \in V$ be a vector not contained in the span of any subset of $n-1$ of the $v_i$.  Let $B$ denote a fixed ordered basis of $V$
 and define for each $i$ the orientation 
 \[ O_B(\hat{v}_i) := O_B(v_0, \dotsc, \hat{v}_i, \dotsc, v_n) = \sign \det(v_0, \dotsc, \hat{v}_i, \dotsc, v_n) \in \{0, \pm 1\}, \] 
where the written  matrix gives the representation of the vectors $v_j$ in terms of the basis $B$, for $j \neq i$.
Then \begin{equation} \label{e:cocycleL}
 \sum_{i=0}^{n} (-1)^i O_B(\hat{v}_i)c_Q(v_0, \dotsc, \hat{v}_i, \dotsc, v_n) \equiv 0 \pmod{\cL_\R}. \end{equation}
Furthermore, if each $v_i$ lies in the totally positive orthant $(\R_{>0})^m$, then in fact
 \[ \sum_{i=0}^{n} (-1)^i O_B(\hat{v}_i) c_Q(v_0, \dotsc, \hat{v}_i, \dotsc, v_n) = 0. \]
\end{theorem}

\begin{proof}  We prove the result by induction on $n$.  For the base case $n=1$, the argument for the ``general position" case below gives the desired result; alternatively one can check the result in this case by hand.

For the inductive step, note first that the result is trivially true by the definition of $c_Q$ unless $\dim V = n$.
We therefore suppose this holds and consider two cases. 

\medskip

\underline{Case 1:}  The $v_i$ are in general position in $V$, i.e.\ any subset of $\{v_0, \dotsc, v_n\}$ of size $n$ spans $V$.  For any $w \in V$, it then follows from our assumption on $Q$ 
that for $\epsilon > 0$ small enough, the set $\{ v_0, \dotsc, v_n, w + \epsilon Q\}$ is in general position in $V$.  In view of the definition
of $c_Q$ given in (\ref{e:cqdef}), Proposition 2 of \cite{hill}
therefore implies that the left side of (\ref{e:cocycleL}) is a constant function on $V$ 
taking the value $d(v_0, \dotsc, v_n)$ defined as follows.
Let $\lambda_i$ for $i = 0, \dotsc, n$ be nonzero constants such that $\sum_{i=0}^{n} \lambda_i v_i = 0$.  The $\lambda_i$ are well-defined up to a simultaneous scalar multiplication.  Then
\begin{equation} \label{e:dhilldef}
d(v_0, \dotsc, v_n) = \begin{cases}
(-1)^i O_B(\hat{v}_i) & \text{if the } \lambda_i \text{ all have the same sign,}\\
0 & \text{otherwise.}
\end{cases} 
\end{equation}
One readily checks that right side of (\ref{e:dhilldef}) is independent of $i$.  Now,  the characteristic function of $V$ lies in $\cL_\R$, giving the desired result. 
Furthermore, if the $v_i$ lie in the totally positive orthant $(\R_{>0})^m$,
then the $\lambda_i$ cannot all have the same sign and hence
$d(v_0, \dotsc, v_n) = 0$.  This completes the proof in the case where the $v_i$ are in general position.

\medskip
\underline{Case 2:}  The $v_i$ are not in general position.  Without loss of generality, assume that $v_0, \dotsc, v_{n-1}$ are linearly dependent. 
Let $V'$ denote the $(n-1)$-dimensional space spanned by these $n$ vectors.
Denote by $\pi'\colon V \rightarrow V'$ and $\pi\colon V \rightarrow \R$
 the projections according to the direct sum decomposition
$V = V' \oplus \R v_n$.
We claim that for  $i = 0, \dotsc, n-1$ and $w \in V$, we have
\begin{equation} \label{e:cqclaim}
 c_Q(v_0, \dotsc, \hat{v}_i, \dotsc, v_n)(w) = c_{\pi'(Q)}(v_0, \dotsc, \hat{v}_i, \dotsc, v_{n-1})(\pi'(w)) \cdot g_Q(w), 
 \end{equation}
 where 
 \[ g_Q(w) = \begin{cases} 1 & \text{if } \pi(w) \ge 0 \text{ and } \pi(w) = 0 \Rightarrow \pi(Q) > 0, \\
 0 & \text{otherwise}.
 \end{cases} \]
 First note that if $v_0, \dotsc, \hat{v}_i, \dotsc, v_n$ are linearly dependent, then under our conditions we necessarily have that 
$v_0, \dotsc, \hat{v}_i, \dotsc, v_{n-1}$ are linearly dependent, and both sides of (\ref{e:cqclaim}) are zero.

Therefore suppose that the vectors $v_0, \dotsc, \hat{v}_i, \dotsc, v_n$ are linearly independent, in which case
$v_0, \dotsc, \hat{v}_i, \dotsc, v_{n-1}$ are clearly linearly independent as well, and hence span $V'$.  Furthermore $\pi'(Q) \in V'$ satisfies the condition that it is not contained in the span of any subset of $n-2$ of these vectors, or else $Q$ would lie in the span of $n-1$ of the original vectors $v_0, \dotsc, v_n$; hence the right side of (\ref{e:cqclaim}) is well-defined.  Equation (\ref{e:cqclaim}) now follows directly from the interpretation of the function $c_Q$ given in (\ref{e:cqinterp}).

To deal with the orientations note that if $B'$ is any other basis of $V$, then \begin{equation}
\label{e:obbp}
 O_B(\hat{v}_i) = O_B(B') \cdot O_{B'}(\hat{v}_i). \end{equation}
We therefore choose for convenience a basis $B'$ for $V$ whose last element is the vector $v_n$.

Using (\ref{e:cqclaim}) and (\ref{e:obbp}) and the fact that $c_Q(v_0, \dots, v_{n-1}) = 0$ since
$v_0, \dots, v_{n-1}$ are linearly dependent, we calculate
\begin{align*}
\sum_{i=0}^{n} (-1)^i O_{B}(\hat{v}_i)c_Q(v_0, \dotsc, \hat{v}_i, \dotsc, v_n)(w) = & \
O_B(B')  \sum_{i=0}^{n-1} (-1)^i O_{B'}(\hat{v}_i)c_Q(v_0, \dotsc, \hat{v}_i, \dotsc, v_n)(w) \\
 = & \ O_B(B') \ell_Q(w) g_Q(w),
\end{align*}
where
\[ \ell_Q(w) =   \sum_{i=0}^{n-1} (-1)^i O_{B'}(\hat{v}_i)c_{\pi'(Q)}(v_0, \dotsc, \hat{v}_i, \dotsc, v_{n-1})(\pi'(w)). \]
Now if we let $B''$ be the basis of $V'$ given by the image of the first $n-1$ elements of $B'$ under $\pi'$, it is clear that
\[ O_{B'}(\hat{v}_i) = O_{B''}(v_0, \dotsc, \hat{v}_i, \dotsc, v_{n-1}).\]
Therefore the function $\ell_Q$ can be written
\[ \ell_Q(w)=  \sum_{i=0}^{n-1} (-1)^i O_{B''}(v_0, \dotsc, \hat{v}_i, \dotsc, v_{n-1})c_{\pi'(Q)}(v_0, \dotsc, \hat{v}_i, \dotsc, v_{n-1})(\pi'(w)). \]
This is the exact form for which we can use the inductive hypothesis to conclude that $\ell_Q \in \cL_\R(V')$ and $\ell_Q = 0$ if each $v_i$ lies in the totally positive orthant.  It is readily checked that this implies that $\ell_Q g_Q \in \cL_\R(V)$ as desired (and $\ell_Q g_Q = 0$ if each $v_i$ lies in the totally positive orthant).
\end{proof}

  \subsection{Signed fundamental domains} \label{s:sfd}

In this section we  show that Theorem~\ref{t:cqcocycle}
can be combined with a result of Colmez to deduce a theorem of Diaz y Diaz and Friedman on the existence of signed Shintani domains.  We use this result in the proof of Theorem~\ref{t:psiszeta} in order to relate our cocycle to the special values of partial zeta functions.

Consider the totally positive orthant $(\R_{>0})^n \subset \R^n$, which forms a group under the operation $*$ of componentwise multiplication.
Let $D = \{ x \in  (\R_{>0})^n: x_1x_2 \cdots x_n = 1 \}$.  Let $U \subset D$ denote a subgroup that is discrete and free of
rank $n-1$.  The goal of this section is to determine an explicit fundamental domain for the action of $U$ on the totally positive orthant
in terms of an ordered basis $\{u_1, \dotsc, u_{n-1}\}$ for $U$.

Define the orientation
\begin{equation} \label{e:wedef}
 w_u := \sign\det( \log(u_{ij}))_{i,j = 1}^{n-1}) = \pm 1, 
 \end{equation}
where $u_{ij}$ denotes the $j$th coordinate of  $u_i$.
 For each permutation $\sigma \in S_{n-1}$ let
\[ v_{i, \sigma} = u_{\sigma(1)} \cdots u_{\sigma(i-1)} \in U, \qquad i = 1, \dotsc, n \]
(so by convention $v_{1, \sigma} = (1, 1, \dotsc, 1)$ for all $\sigma$).
Define
\[ w_\sigma = (-1)^{n-1} w_u \sign(\sigma) \sign(\det(v_{i,\sigma})_{i=1}^{n}) \in \{0, \pm 1\}. \]

We choose for our perturbation vector the coordinate basis vector
 $e_n = (0, 0, \dotsc, 0, 1)$, and assume that $e_n$ satisfies the property that it
 does not lie in the $\R$-linear span of any $(n-1)$ of elements of $U$. 
Note that the action of $U$ preserves the ray  $\R_{>0}e_n.$

\begin{theorem}[Colmez, \cite{colmezresidue}, Lemme 2.2]   \label{t:colmez}
If $w_\sigma = 1$ for all $\sigma \in S_{n-1}$, then
\begin{equation} \label{e:colmez}
\bigsqcup_{\sigma \in S_{n-1}} C_{e_n}(v_{1, \sigma}, \dotsc, v_{n, \sigma}) 
\end{equation}
is a fundamental domain for the action of $U$ on the totally positive orthant $(\R_{>0})^n$.  In other words, we have
\[ \sum_{u \in U} \sum_{\sigma \in S_{n-1}} c_{e_n}(v_{1, \sigma}, \dotsc, v_{n, \sigma})(u * x) = 1 \]
for all $x \in (\R_{>0})^n$.
\end{theorem}

\begin{remark} \label{r:posorthant}
Note that each of the vectors $v_{i,\sigma}$ lies in the positive orthant, so
each open cone $C(v_{i_1, \sigma}, \dotsc, v_{i_r, \sigma})$ is contained in the positive orthant when $r \ge 1$.  Furthermore, 
$e_n$ lies along a coordinate axis and is not contained in $ C(v_{1, \sigma}, \dotsc, v_{n, \sigma}) $, hence $0 \not \in
 C_{e_n}(v_{1, \sigma}, \dotsc, v_{n, \sigma}) $.  Therefore $C_Q(v_{1, \sigma}, \dotsc, v_{n, \sigma}) \subset (\R_{>0})^n$.
\end{remark}

The following generalization was recently proved by Diaz y Diaz and Friedman using topological degree theory.  We will 
show that the cocycle property of $c_Q$ proved in Theorem~\ref{t:cqcocycle} allows one to deduce their theorem from the earlier result of Colmez.
Note that our proof of the theorem relies upon Colmez's theorem, whereas the proof of Diaz y Diaz and Friedman recovers it.

\begin{definition} \label{d:sfd}
A {\em signed fundamental domain}
for the action of $U$ on $(\R_{>0})^n$ is by definition a formal linear combination $D = \sum_i a_i C_i$ of
open cones with $a_i \in \Z$ such that $\sum_{u \in U} \sum_i a_i {\mathbf 1}_{C_i}(u *x) = 1$ 
for all $x \in (\R_{>0})^n$.  We call  ${\mathbf 1}_D := \sum_i a_i {\mathbf 1}_{C_i} \in \cK_\R$ the characteristic function of $D$.
\end{definition}

Note that when each $a_i = 1$ and the $C_i$ are disjoint, the set $\sqcup_{i} C_i$ is a fundamental domain in the usual sense.

\begin{theorem}[Diaz y Diaz--Friedman, \cite{ddf}, Theorem 1]   \label{t:ddf}
The formal linear combination
\[ \sum_{\sigma \in S_{n-1}} w_\sigma C_{e_n}(v_{1, \sigma}, \dotsc, v_{n, \sigma})\]
is a signed fundamental domain for the action of $U$ on $(\R_{>0})^n$, i.e.\ 
\begin{equation} \label{e:diaz}
 \sum_{u \in U} \sum_{\sigma \in S_{n-1}}  w_\sigma c_{e_n}(v_{1, \sigma}, \dotsc, v_{n, \sigma})(u * x) = 1 
 \end{equation}
for all $x \in (\R_{>0})^n$.
\end{theorem}

\begin{proof}  Colmez proved the existence of a finite index subgroup $V \subset U$ such that the condition $w_\sigma = 1$ for all $\sigma$ holds for some basis of $V$ (see~\cite{colmezresidue}, Lemme 2.1).  Fix such a subgroup $V.$
Our technique is to reduce the desired result for $U$ to the result for $V$, which is given by Colmez's theorem.

Endow the abelian group $\cK_\R$ with an action of $U$ by \begin{equation} \label{e:uaction}
(u f)(x) = f(u^{-1} * x). \end{equation} 
The key point of our proof is the construction of a cohomology class $[\phi_U] \in H^{n-1}(U, \cK_\R)$ as follows.
Given $v_1, \dotsc, v_n \in U$,   let
\begin{equation}
\label{e:tphidef}
 \phi_U(v_1, \dotsc, v_n) = \sign(\det( v_i )_{i=1}^{n}) c_{e_n}(v_1, \dotsc, v_n) \in \cK_\R. 
 \end{equation}
The $U$-invariance of $\phi_U$ follows from the definition of $c_{e_n}$ given in (\ref{e:cqdef})
along with the above-noted property that the action of $U$ preserves $\R_{>0}e_n.$
The fact that $\phi_U$ satisfies the cocycle property
\[ \sum_{i=0}^{n} (-1)^i\phi_U(v_0, \dotsc, \hat{v}_i, \dotsc, v_n) = 0 \]
 is given by Theorem~\ref{t:cqcocycle}, since the $v_{i}$ 
 lie in
the positive orthant. 
  We let $[\phi_U] \in H^{n-1}(U, \cK_\R)$ be the cohomology class  represented by the homogeneous cocycle $\phi_U$.

The basis $u_1, \dotsc, u_{n-1}$ of $U$ gives an explicit element $\alpha_U \in H_{n-1}(U, \Z) \cong \Z$ as 
follows. 
We represent homology classes by the standard projective resolution $C^*(U) = \Z[U^{*+1}]$ of $\Z$, and let
$\alpha_U$ be the class represented by the cycle
\begin{equation} \label{e:alphadef}
 {\alpha}(u_1, \dotsc, u_{n-1}) = (-1)^{n-1} w_u \sum_{\sigma \in S_{n-1}} \sign(\sigma) [(v_{1, \sigma}, \dotsc, v_{n, \sigma})] \in \Z[U^n]. 
 \end{equation}
It is a standard calculation that $d{\alpha}(u_1, \dotsc, u_{n-1}) = 0$ and
that  the cohomology class  $\alpha_U$
represented by ${\alpha}(u_1, \dotsc, u_{n-1}) $
depends only on $U$ and not the chosen basis $u_1, \dotsc, u_{n-1}$ (see \cite[Lemma 5]{sczech}).

The image of $([\phi_U], \alpha_U)$ under the cap product pairing
\[ H^{n-1}(U, \cK_\R) \times H_{n-1}(U, \Z) \longrightarrow \cK_{\R, U} := H_0(U, \cK_{\R}) \]
is by definition the image of the function $ \sum_{\sigma \in S_{n-1}}  w_\sigma c_Q(v_{1, \sigma}, \dotsc, v_{n, \sigma})$ in $\cK_{\R, U}$.

  Let $\cJ$ denote the group of functions $(\R_{>0})^n \longrightarrow \Z$, which is endowed with an action of $U$
as in (\ref{e:uaction}).  Denote by $\Sigma_U: \cK_{\R,U} \rightarrow \cJ^U$ the map defined by
 \begin{equation} \label{e:sigmau}
  (\Sigma_U f)(x) = \sum_{u \in U} f(u * x). 
  \end{equation} 
Note that the sum (\ref{e:sigmau}) is locally finite by the following standard compactness argument.  The action of $U$ preserves the
product of the coordinates of a vector, and applying $\log$ to the coordinates sends
the surface  $\{x_1 \cdots x_n = $ constant$\}$ to a hyperplane.  In this hyperplane,
  the image of a cone is bounded, and the action of $U$ is translation by a lattice. Given a point $x$, 
  only finitely many lattice points can translate $x$ into the bounded region corresponding to a cone. 

Now $\Sigma_U(\phi_U \cap \alpha_U) \in \cJ^U$ is by definition the function on the left side of (\ref{e:diaz}), namely
\[  \sum_{u \in U} \sum_{\sigma \in S_{n-1}}  w_\sigma c_{e_n}(v_{1, \sigma}, \dotsc, v_{n, \sigma})(u * x). 
\]

It remains  to analyze this picture when $U$ is replaced by its finite index subgroup $V$ chosen at the outset of the proof.
General properties of group cohomology (see [Br, pp. 112--114])  yield a commutative diagram:

\[ 
\xymatrix{
H^{n-1}(V, \cK_\R) \times H_{n-1}(V, \Z) \ar^(.73){\cap}[r] 
 \ar@<5ex>^{\cores}[d] & \cK_{\R, V} \ar[d] \ar^{\Sigma_V}[r] & \cJ^V \ar^{\Sigma_{U/V}}[d] \\
H^{n-1}(U, \cK_\R) \ar@<5ex>^{\res}[u] \times H_{n-1}(U, \Z) \ar^(.73){\cap}[r] & \cK_{\R, U} \ar^{\Sigma_U}[r] & \cJ^U.
}
\]
Here $\Sigma_{U/V} : \cJ^V \rightarrow \cJ^U$ is given by 
 \[ (\Sigma_{U/V} f)(x) = \sum_{u \in U/V} f(u * x). \] 

The desired result now follows from the fact that $\Sigma_V(\phi_V \cap \alpha_V) = \1 := \1_{(\R_{>0})^n}$ by Colmez's theorem, along 
with
\begin{align}
\cores(\alpha_V) = & \ [U: V] \cdot \alpha_U, \label{e:uvcores} \\
 \res{\phi_U} = & \ \phi_V, \label{e:uvres} \\ 
\Sigma_{U/V}(  \1) = & \ [U: V]  \cdot  \1. \label{e:uvsigma}
\end{align}
Equation (\ref{e:uvcores}) is proven in   \cite[Sect. III, Prop. 9.5]{brownbook},
whereas (\ref{e:uvres}) and (\ref{e:uvsigma}) are obvious.
\end{proof}

\subsection{The Shintani cocycle on $\GL_n(\Q)$}

Recall the notation $\Gamma = \GL_n(\Q)$.  In this section we define a Shintani cocycle $\Phi_{\mathrm{Sh}}$ on $\Gamma$.
 This cocycle will be directly related to the cocycles $\phi_U$ defined in the previous section; however, since our cocycle will be defined on the full group $\Gamma$ rather than the simpler groups $U \subset D$ in the positive orthant, we will need to consider the quotient  $\cK_\R/\cL_\R$ rather than  $\cK_\R$ (cf.\ definition (\ref{e:ldef}) and the appearance
of $\cL_\R$ in Theorem~\ref{t:cqcocycle}).  The relationship between $\Phi_{\mathrm{Sh}}$ and the $\phi_U$ in our cases of interest will be stated precisely in Section~\ref{rs:specialize} below.

\bigskip

 Let  $\R^n_{\Irr}$ denote the set of elements in $\R^n$ (viewed as row vectors) whose components are linearly independent over 
 $\Q$, i.e. the set of vectors $Q$ such that $Q \cdot x \neq 0$ for nonzero $x \in \Q^n$. The set $\R^n_{\Irr}$ is a right $\Gamma$-set by the action of right multiplication; we turn this into a left action by multiplication on the right by the transpose (i.e.\ $\gamma \cdot Q := Q \gamma^{t}$).  Note  that any $Q \in \R^n_{\Irr}$ satisfies the property that it does not lie in the $\R$-linear span of any $n-1$ vectors in $\Q^n \subset \R^n$.  The elements of $\R^n_{\Irr}$ will therefore serve as our set of auxiliary perturbation vectors as employed in Section~\ref{s:qperturb}.\symbolfootnote[2]{To orient the reader who may be familiar with the
notation of  \cite{colmez} or \cite{ddf} in which one takes $Q = e_n = (0, 0, \dotsc, 0, 1)$ as in Section~\ref{s:sfd}, one goes from this vector to an element of our $\R^n_{\Irr}$
by applying a change of basis given by the image in $\R^n$ of a basis of a totally real field $F$ of degree $n$.  Our notation allows for
rational cones $C$ and irrational perturbation vectors $Q$ rather than the reverse.  This is convenient for comparison with Sczech's cocycle,
in which one also chooses a  vector $Q \in \R^n_{\Irr}$.  See Section~\ref{rs:specialize} and in particular (\ref{e:qdef}) for more details.}
  We let $\cQ = \R^n_{\Irr}/\R_{>0}$, the set of equivalence class of elements of $\R^n_{\Irr}$ under
 multiplication by positive reals. 

Let $\cK \subset \cK_\R$ denote the subgroup generated by the characteristic functions of rational open cones, i.e.\ by the
characteristic functions of cones $C(v_1, \dotsc, v_n)$ with each $v_i \in \Q^n$.  Let $\cL  = \cL_\R \cap \cK$ with $\cL_\R$ as in
(\ref{e:ldef}).
The abelian group $\cK$ is naturally endowed with a left $\Gamma$-module structure via
\[ \gamma \cdot \varphi (x) = \sign(\det \gamma)) \varphi(\gamma^{-1}x), \]
and $\cL$ is a $\Gamma$-submodule of $\cK$.

Let $\cN$ denote the abelian group of maps $\cQ \lra \cK/ \cL$.  This space is endowed with a $\Gamma$-action 
given by $(\gamma f)(Q) = \gamma f( \gamma^{-1} Q)$.
We now define a homogeneous cocycle  \[ \Phi_{\mathrm{Sh}} \in Z^{n-1}(\Gamma, \cN). \]

For $A_1, \dotsc, A_n \in \Gamma$, let $\sigma_i$ denote the first column of $A_i$.  Given  $Q \in \cQ$,  define
\begin{equation} \label{e:phisdef}
 \Phi_{\mathrm{Sh}}(A_1, \dotsc, A_n)(Q) = \sign(\det(\sigma_1, \dotsc, \sigma_n))c_Q(\sigma_1, \dotsc, \sigma_n) 
 \end{equation}
with $c_Q$ as in (\ref{e:cqdef}).

\begin{theorem} \label{t:phis}  We have $\Phi_{\mathrm{Sh}} \in Z^{n-1}(\Gamma, \cN).$
\end{theorem}

\begin{proof}
 The fact that $\Phi_{\mathrm{Sh}}$ is $\Gamma$-invariant follows directly from the definitions.  In $\mc K,$ the cocycle property  
\[ \sum_{i=0}^{n} (-1)^i \Phi_{\mathrm{Sh}}(A_0, \dotsc, \hat{A}_i, \dotsc, A_n)(Q) \equiv 0 \pmod{\cL} \] follows  from Theorem~\ref{t:cqcocycle} using for $B$ the standard basis of $\R^n.$
\end{proof}
 
Denote by $[\Phi_{\mathrm{Sh}}] \in H^{n-1}(\Gamma, \cN)$  the cohomology class represented by the homogeneous cocycle $\Phi_{\mathrm{Sh}}$.

\section{Applications to Zeta Functions} \label{s:zeta}

\subsection{Totally real fields} \label{rs:specialize} 

Let $F$ be a totally real field of degree $n$, and denote by $J_1, \dotsc, J_n: F \rightarrow \R$ the $n$
real embeddings of $F$.  Write $J = (J_1, \dotsc, J_n): F \rightarrow \R^n$.
 We denote the action of $F^*$ on $\R^n$ via composition with $J$ and componentwise multiplication by $(x, v) \mapsto x * v$.  Let $U$ denote a subgroup of finite index in the group of totally positive units in $\cO_F^*$.  We can apply the discussion
 of Section~\ref{s:sfd} on fundamental domains to the group $J(U) \subset D$.
 
  Note that  $e_n = (0,0, \dotsc, 0, 1)$ satisfies the property that
it does not lie in the $\R$-linear span of any $n-1$ elements of the form $J(u)$ for $u \in F^*$.  Indeed, given $u_1, \dotsc, u_{n-1} \in F^*$,
there exists an $x \in F^*$ such that $\Tr_{F/\Q}(x u_i) =0$ for all $i = 1, \dotsc, n-1$.  
Dot product with $J(x)$ defines an $\R$-linear functional on $\R^n$ that vanishes on the $J(u_i)$ but not on $e_n$, proving the claim.

In this section we explain the relationship between the class $[\Phi_{\mathrm{Sh}}]$ and 
the class $[\phi_U]$ defined in Section~\ref{s:sfd} (where we write $\phi_U$ for $\phi_{J(U)}$). 
 Choosing a $\Z$-basis $w = (w_1, w_2, \dotsc, w_n)$ of $\cO_F$ yields an embedding $\rho_w\colon F^* \rightarrow \Gamma$ given by
\begin{equation}   (w_1u, w_2u, \dotsc, w_nu) =  (w_1, w_2, \dotsc, w_n) \rho_w(u). \label{e:rhowdef} \end{equation}
Pullback by $\rho_w$ (i.e.\ restriction) yields a class $\rho_w^* \Phi_{\mathrm{Sh}} \in H^{n-1}(U, \cN)$.

Denote by $J(w) \in \GL_n(\R)$ the matrix given by $J(w)_{ij} = J_i(w_j)$.  Note that if we let $\diag(J(u))$  be the diagonal matrix with diagonal entries $J_i(u)$, then 
\begin{equation} \label{e:diag}
\rho_w(u) = J(w)^{-1} \diag(J(u)) J(w). 
\end{equation}
Let \begin{equation} \label{e:qdef}
Q = (0, 0, \dotsc, 1)J(w)^{-t}. 
\end{equation}
The vector $Q$ is the image under $J_n$ of the dual basis to $w$ under the trace pairing $F \times F \rightarrow \Q$, $(x, y) \mapsto \Tr_{F/\Q}(xy)$.
In particular, $Q$ is an element of $\R^n_{\Irr}$. Furthermore, (\ref{e:diag}) and (\ref{e:qdef}) yield  \[ Q \rho_w(x)^t = J_n(x) Q \] for $x \in F^*$,
which implies that
the image of $Q$ in $\cQ$ is invariant under the action of $U$.  We can therefore view $Q$ as an element of
$H^0(U, \Z[\cQ])$.  In conjunction with the canonical map $\cN \times \cQ \rightarrow \cK/\cL$ given by $(f, Q) \mapsto f(Q)$,
the cup product gives a map \[ H^{n-1}(U, \cN) \times H^0(U, \Z[\cQ]) \rightarrow H^{n-1}(U, \cK/\cL) \]
yielding an element $\rho_w^* [\Phi_{\mathrm{Sh}}] \cup Q \in H^{n-1}(U, \cK/\cL)$.

Now consider the map induced by $J(w)$, denoted
 \[ J(w)^*: \cK_\R \lra \cK_\R, \] given by $(J(w)^*f)(x) =  f(J(w)x)$.  Our desired relation is 
 \[ J(w)^* [\phi_U] = \rho_w^* [\Phi_{\mathrm{Sh}}] \cup Q \]
  in $H^{n-1}(U, \cK_\R/\cL_\R).$ In fact, this relationship holds on the level of cocycles as follows.
  For any $x \in F^*$, we define a modified cocycle $\Phi_{\Sh, x} \in Z^{n-1}(\Gamma, \cN)$ by letting $\gamma = \rho_w(x)^{-1}$ and setting
  \begin{equation} \label{e:phishx}
   \Phi_{\Sh, x}(A_1, \dotsc, A_n) = \Phi_{\Sh}( A_1 \gamma, \dotsc,  A_n \gamma). 
   \end{equation}
It is a standard fact in group cohomology that 
the cohomology class represented by ${\Phi}_{\Sh, x}$ is independent of $x$ and hence equal to $[\Phi_{\Sh}]$
(see \cite[Lemma 4]{sczech}).
We have the following equality of cocycles:
\[
 J(w)^* {\phi}_{U} = \rho_w^* \Phi_{\mathrm{Sh}, w_1} \cup Q
\]
in $Z^{n-1}(U, \cK_\R/\cL_\R)$. 
In concrete terms, this says for $u = (u_1, \dotsc, u_n)$:
\begin{equation}  \label{e:jw}
 \phi_U(u)(J(w)x) = \Phi_{\Sh, w_1}(\rho_w(u), Q)(x).
\end{equation}
 In Section~\ref{s:svz} this relationship will be used along with Theorem~\ref{t:ddf} to relate
the class $[\Phi_{\mathrm{Sh}}]$  to special values of zeta functions attached to the field $F$.
Over the next few sections we first we recall Shintani's results on cone zeta functions.

\subsection{Some bookkeeping} \label{s:bookkeeping}

We will be interested in sums over the points lying in the intersection of open simplicial cones with certain lattices in $\R^n$.
In this section we introduce a convenient way of enumerating these points.
Let $\cV = \Q^n/\Z^n$, and consider for $v \in \cV$ the associated lattice $v + \Z^n \subset \R^n$.

Let $C$ be a rational open cone.  By scaling the generators of $C$,
we can find $\R$-linearly independent vectors $\sigma_1,\ldots,\sigma_r\in \Z^n$ such that
$
C=\R_{>0}\sigma_1+\cdots + \R_{>0}\sigma_r.
$
Let $\sP = \sP(\sigma_1, \dotsc, \sigma_r)$ denote the half-open parallelpiped generated by the $\sigma_i$:
\begin{equation} \label{e:pdef}
\sP = \left\{x_1\sigma_1+\cdots + x_r\sigma_r : 0 < x_1,\ldots,x_r \le 1\right\},
\end{equation}
with the understanding that $\sP(\emptyset) = \{0\}$ in the case $r=0$.
Then
\begin{equation} \label{e:cdecomp}
C \cap (v + \Z^n) =\bigsqcup_{a\in \sP \cap (v + \Z^n)}(a + \Z_{\geq 0} \sigma_1+\cdots+\Z_{\geq 0} \sigma_r),
\end{equation}
where the disjointness of the union follows from the linear independence of the $\sigma_i$.

Now let $C$ be a rational open cone of maximal dimension $r= n$ in $\R^n$.
Let $Q \in \cQ$ and consider the set $C_Q$ defined in Section~\ref{s:qperturb}, consisting of the disjoint union $C$ 
and some of its boundary faces of all dimensions.  We would like to enumerate the points in $C_Q \cap (v + \Z^n)$.

For each subset $I \subset \{1, \dotsc, n\}$, the boundary face $C_I = C(\sigma_i : i \in I)$ is assigned a weight
via the $Q$-perturbation process denoted $\weight(C_I) \in \{ 0, 1\}$ and  given by (\ref{e:weight}).
Associated to each  cone $C_I$ is the parallelpiped $\sP_I = \sP(\sigma_i: i \in I)$.  We 
have
\begin{equation}
\label{e:cqdecomp0}
 C_Q \cap (v + \Z^n) =\bigsqcup_{\stack{I \subset \{1, \dotsc, n\}}{a\in \sP_I \cap (v + \Z^n)}} \weight(C_I) (a + \sum_{i \in I} \Z_{\ge 0} \sigma_i),
\end{equation}
where our notation means that the set $(a + \sum \Z_{\ge 0} \sigma_i)$ should be included
if $\weight(C_I) = 1$ and not included if $\weight(C_I) = 0$.

Let $\sigma \in M_n(\Z)  \cap \Gamma$ denote the matrix whose columns are the $\sigma_i$.
For each $a \in \sP_I \cap (v + \Z^n)$ that occurs as $I$ ranges over all subsets of $\{1, \dotsc, n\}$, 
we can associate the class $x = a - v \in \Z^n/\sigma \Z^n$.
Conversely, given a class $x \in \Z^n/\sigma \Z^n$, there will be at least one $a$ giving rise to that class. 

To be more precise, let $J = J(x)$ denote the set of indices $j$ for which $(\sigma^{-1}(v+x))_j \in \Z$.  The number of points $a$
giving rise to the class $x$ is $2^{\# J}$. Let $\overline{J} = \{1, \dotsc, n\}  - J$. 
For each $I \supset \overline{J}$, we can write down a unique point $a_I \in \sP_I$ such
that the image of $a_I -v$ in $\Z^n/\sigma \Z^n$ is equal to $x$. 
We define $a_I$ by letting $\sigma^{-1}(a_I)$ be congruent to $\sigma^{-1}(v + x)$ modulo $\Z^n$, and 
further requiring $\sigma^{-1}(a_I)_i \in (0, 1)$ if $i \not \in J$, and
\begin{equation}\label{e:aIdef} \sigma^{-1}(a_I)_i := \begin{cases} 0 & i \in J \cap \overline{I} = \overline{I} \\
1 & i \in J \cap I. 
\end{cases}
\end{equation}
We can then rewrite (\ref{e:cqdecomp0}) as
\begin{equation} \label{e:cqdecomp}
 C_Q \cap (v + \Z^n) =\bigsqcup_{\stack{x \in \Z^n/\sigma \Z^n}{I \supset \overline{J(x)}}}
\weight(C_I)(a_I + \sum_{i \in I} \Z_{\ge 0} \sigma_i).
\end{equation}
This decomposition will be used in Sections~\ref{s:genscz} and~\ref{s:explicit}.

\subsection{Cone generating functions}
Let $C$ be a rational open cone in $\R^n$ and let $v\in\Q^n$.  Let $x_1,\ldots,x_n$ be variables and let $g(C,v)$ be the generating series for the set of integer points in $C-v$:
\[
g(C,v)(x)=\sum_{m\in (C-v)\cap \Z^n} x^m\in \Q[[x,x^{-1}]],
\]
where as usual $x^m$ denotes $x_1^{m_1} \cdots x_{n}^{m_n}$.
If $\mu\in\Z^n$, then $
g(C,v+\mu)=x^{-\mu}g(C,v).$

The series $g(C,v)$ is actually the power series expansion of a rational function.  In fact, the decomposition (\ref{e:cdecomp})  gives rise to the identity
\begin{equation} \label{e:shdef}
 g(C,v)(x) = \frac{\sum_{a \in (\sP-v) \cap\Z^n} x^a }{(1 - x^{\sigma_1}) \cdots (1 - x^{\sigma_r})}\in\Q(x),
\end{equation}
where $\sigma_i$ are integral generators of the cone $C$, and $\sP = \sP(\sigma_1, \dotsc, \sigma_r)$ 
is the half-open parallelpiped defined in (\ref{e:pdef}).

 Write $c$ for the characteristic function of $C$ and define $g(c,v)=g(C,v)$.
The following fundamental algebraic result was proved independently by Khovanskii and Pukhilov~\cite{KP} and Lawrence~\cite{Lawrence} (cf.~\cite[Theorem 2.4]{bar}).
\begin{prop} There is a unique map $g:\cK \times\Q^n\lra \Q(x)$ that is
$\Q$-linear in the first variable such that
$g(c,v)=g(C,v)$ for all rational open cones $C$ 
and $g(c,v)=0$ if $c\in\cL$.
\end{prop}
\noindent Thus we may view $g$ as a pairing
\[
g:\cK/\cL\times \Q^n\lra \Q(x).
\]
Let $\Q((z))$ be the field of fractions of the power series ring $\Q[[z]]$.
In our applications, we will consider images of the functions $g(C,v)$ under the mapping $\Q(x)\to \Q((z))$ defined by $x_i \mapsto e^{z_i}$.  Define 
\[
h(C,v)(z)=e^{v\cdot z}g(C,v)(e^{z_1},\ldots,e^{z_n})\in \Q((z)).
\]
With $\sigma_1,\ldots,\sigma_r$ and $\sP$  as above, we have
\[
h(C,v)(z)=\frac{\sum_{a \in \sP \cap (v + \Z^n)} e^{a\cdot z} }{(1 - e^{\sigma_1\cdot z}) \cdots (1 - e^{\sigma_r\cdot z})}.
\]
  From the corresponding properties of the functions $g(C,v)$, it follows immediately that $h$ may be viewed as a pairing
\[
h:\cK/\cL\times \Q^n/\Z^n\lra \Q((z))
\]
that is linear in the first variable.  
We call $h$ the \emph{Solomon--Hu} pairing owing to its first appearance in the works~\cite{So,husolomon}.

\subsection{Special values of Shintani zeta functions} \label{s:svszf}
We now recall results relating the generating function $g(C,v)$ introduced above to special values of complex analytic  {\em Shintani zeta functions}, whose definition we now recall.

Let $\cM \subset M_n(\R)$ be the subset of matrices
such that the entries of each column are linearly independent over $\Q$ (i.e.\ for each nonzero row vector $x \in \Q^n$ and $M \in \cM$, the vector $xM$ has no component equal to 0).  
Let $\cD \subset \SL_n(\R)$ be the subgroup of $n \times n$ real diagonal matrices with determinant 1.
Given $M \in \cM$, define a polynomial $f_M \in \R[x_1, \dots, x_n]$ by
\begin{align}
 f_M(x_1, \dotsc, x_n) &= \N((x_1, \dotsc, x_n)M) \nonumber \\
 & = (xM)_1 (xM)_2 \cdots (xM)_n.  \label{e:fmdef}  
 \end{align}
Note that  $f_M$ depends only on the image of the matrix $M$ in $\cM/\cD$.

View the elements of the rational open cone $C=C(w_1,\ldots,w_r) \subset \R^n$ as column vectors.
Choose the $w_i$ to have integer coordinates.
We consider a matrix $M\in\cM$ such that $(C,M)$ satisfies the following positivity condition:   
\begin{equation}\label{E:Mpos}
M^t w \subset (\R_{>0})^n \text{ for all } w \in C.
\end{equation}

This positivity condition will be needed when defining analytic Shintani zeta functions.  When dealing with their algebraic incarnations, i.e.\ the cone generating functions $h(C,v)$  introduced in the previous section, it is not required.  This added flexibility in the algebraic setting is crucial for the cohomological constructions to be described in the following sections.
With $C$ and $M$ as above and  a vector  $v \in \cV$, define the Shintani zeta function
\[ \zeta(C, M, v, s) = \sum_{x \in C \cap v + \Z^n} \frac{1}{f_M(x)^s}. \]
Using~\eqref{E:Mpos}, it is easy to see that this series is absolutely convergent for $s\in \C$ with $\real(s)>1$.  
Letting $\sP=\sP(w_1,\ldots,w_r)$ be the parallelpiped defined in (\ref{e:pdef}) and  $W = (w_1, \dotsc, w_r)$ the $n \times r$ matrix whose columns are the 
generators of the cone $C$, we define
\[
Z(C,M,a,s)=\sum_{x\in(\Z_{\geq 0})^r}\frac{1}{f_M(a+Wx)^s}
\]
for $a\in\sP\cap (v+\Z^n)$.
We obtain the finite sum decomposition 
\[
\zeta(C,M,v,s)=\sum_{a\in \sP\cap v+\Z^n}Z(C, M,a,s).  
\]
Shintani~\cite{Sh} proved that each $Z(C,M, a, s)$, and hence $\zeta(C,M,v,s)$ itself, admits a meromorphic continuation to $\C$.  

Shintani also gave a formula for the values of these zeta functions at nonpositive integers.  Observe that if $k$ is a nonnegative integer, then $f_M(x)^k$ is $(k!)^n$ times the coefficient of $\N(z)^k$ in the Taylor series expansion of $e^{zM^tx}$.  Summing,  we obtain the nonsense identity chain
\begin{multline*}
\text{``} \zeta(C,M,v,-k)=\sum_{x\in C\cap(v+\Z^n)} f_M(x)^k = (k!)^n\coeff\left(
\sum_{x\in C\cap(v+\Z^n)} e^{zM^tx}, \N(z)^k\right)\\
=(k!)^n\coeff\left(h(C,v)(zM^t),\N(z)^k\right)\text.{"}
\end{multline*}
Almost nothing in the above identity chain is actually defined
and in particular the given sums do not converge.  Further, $h(C,v)(zM^t)$ is not holomorphic on a punctured neighborhood of $z=(0,\ldots,0)$ if $n>1$, making the notion of coefficient undefined.  Nonetheless, via an algebraic trick---really, an algebraic version of the trick used by Shintani in his proof of the analytic continuation of $\zeta(C,M,v,s)$---we generalize the notion of coefficient to a class of functions including the $h(C,v)(zM^t)$.  Remarkably, with this generalized notion of coefficient, the identity
\begin{equation} \label{e:shintanibogus}
\zeta(C,M,v,-k)=(k!)^n\coeff\left(h(C,v)(zM^t),\N(z)^k\right)
\end{equation}
holds.  We now define Shintani's operator and state his theorem giving a rigorous statement of (\ref{e:shintanibogus}).
 
Let $K$ be a subfield of $\C$. For $1\leq j\leq n$, we write
\begin{equation} \label{e:zjdef}
Z_j = (z_jz_1,\ldots,z_jz_{j-1},z_j,z_jz_{j+1},\ldots,z_jz_n).
\end{equation}
The following lemma is elementary.
\begin{lemma}\label{l:ppas}
Let $g \in K[[z]]$, let $p\in K[z]$ be homogeneous of degree $d$ with $\coeff(p,z_j^d)\neq 0$, and let $G=g/p$.  Then $G(Z_j)\in z_j^{-d}K[[z]]$.
\end{lemma}
Call a homogeneous polynomial $p\in K[z]$ of degree $d$ \emph{powerful} if the power monomials in $p$ all have nonzero coefficients, i.e., if $\coeff(p,z_j^d)\neq 0$ for all $j$.  The powerful polynomials of interest to us arise as follows.  Call a linear form $L(z)=\ell_1z_1+\cdots+\ell_nz_n$ \emph{dense} if $\ell_j\neq 0$ for all $j$.  If $L_1,\ldots,L_r$ are dense linear forms, then $p=L_1\cdots L_r$ is powerful.

\begin{definition}
 Let $K((z))^{\hd}\subset K((z))$ be the subalgebra consisting of   $G\in K((z))$ that can be written in the form $G=g/p$ for a power series $g\in K[[z]]$ and a powerful homogeneous polynomial $p\in K[z]$.
\end{definition}

\begin{lemma}\label{L:hd} Suppose $C$ is a rational open simplicial cone in $\R^n$, $M\in \cM$ and $v\in\Q^n$.  Let $\Q(\{m_{ij}\})$ be the field generated by the entries of $M$.  Then $h(C,v)(zM^t) \in \Q(\{m_{ij}\})((z))^{\hd}$.
\end{lemma}
\begin{proof}  Write $C = C(w_1, \dotsc, w_r)$
and let $a\in \sP\cap(v+\Z^n)$.  Then $a\in\Q^n$, so $e^{zM^ta}\in K[[z]]$.  For each $j=1,\ldots,r$, set $L_j(z)=zM^tw_j$.  Then we can write $1-e^{zM^tw_j}=L_j(z)g_j(z)$ with $g_j\in K[[z]]^\times$.  Setting $f_a=e^{zM^ta}g_1^{-1}\cdots g_r^{-1}$ and $p=L_1\cdots L_r$, we have \[ h(C,v)(zM^t)= \sum_{a \in \sP \cap(v+\Z^n)} f_a/p.\]  It remains to show that $p$ is powerful.
Since $M\in \cM$ and $w_j\in \Q^n$ for all $j$, it follows that each $L_j$ is dense.  Therefore $p$ is powerful as desired.
\end{proof}
By Lemma~\ref{l:ppas}, if $G\in K((z))^{\hd}$, then $\coeff( G(Z_j),z^m)$ makes sense for any $j$ and any $m\in\Z^n$. This leads to the following definition.

\begin{definition} For $j =1, \dotsc, n$, define operators $\Delta^{(k)}_j:K((z))^{\hd}\to K$ by
 \begin{equation} \label{e:deltai}
 \Delta^{(k)}_j G =  \coeff(G(Z_j), \N(Z_j)^k), 
 \end{equation}
where $Z_j$ is given in (\ref{e:zjdef}).
 Define the {\em Shintani operator} $\Delta^{(k)}:K((z))^{\hd}\to K$ by
\begin{equation}\label{averagedelta}\Delta^{(k)} = \frac{(k!)^n}{n} \sum_{j=1}^{n} \Delta^{(k)}_j. \end{equation}
\end{definition}

\begin{remark} \label{r:shreg}
If $g \in K[[z]]$, then $\Delta^{(k)} g$ is simply $(k!)^n$ times the coefficient of $(z_1 \cdots z_n)^k$ in $g$.  Thus, the operator $\Delta^{(k)}$ extends the coefficient extraction operation from $K[[z]]$ to $K((z))^{\hd}$.
\end{remark}

The Shintani operator shares the following properties with the operation of taking the $(z_1\cdots z_n)^k$-coefficient of 
a regular power series.  The proof is  an elementary computation.

\begin{lemma}\label{L:indepofD}
Let $h\in K((z))^{\hd}$.  Then 
\begin{itemize} \item For  $d_1,\ldots,d_n\in K$, we have
$\Delta^{(k)}h(d_1z_1,\ldots,d_nz_n) = (d_1\cdots d_n)^k \Delta^{(k)}h(z_1,\ldots,z_n).$
\item For any permutation $\sigma$, we have
$\Delta^{(k)}h(z_{\sigma(1)},\ldots,z_{\sigma(n)}) =  \Delta^{(k)}h(z_1,\ldots,z_n).$
\end{itemize}
\end{lemma}

Finally, we may state the following theorem of Shintani:
\begin{theorem}[{\cite[Proposition 1]{Sh}}]  \label{t:shintani}
Let $C$ be a rational open cone,  $v \in \cV$, and $M\in \cM$ satisfying~\eqref{E:Mpos}.  
The function $\zeta(C, M, v, s)$ has a meromorphic continuation to $\C$ and satisfies
\[ \zeta(C , M , v,  -k) = \Delta^{(k)}h(C,v)(zM^t) \qquad \text{ for } k \in \Z_{\ge 0}.
\]
\end{theorem}
We observe that by  Lemma~\ref{L:indepofD}, the coefficient $\Delta^{k}h(C,v)(zM^t)$ depends only on the image of $M$ in $\cM/\cD$. 

\subsection{The power series-valued Shintani cocycle} \label{s:pscoc}

In this section we define the Shintani cocycle in the form that will be most useful for our desired applications; in particular, the cocycle will take values in a module $\cF$ for which it can be compared to the Eisenstein cocycle defined by Sczech in \cite{sczech} and studied in \cite{pcsd}.

\bigskip

  The set $\cM$ defined in  Section~\ref{s:svszf} is  naturally a left $\Gamma$-set via the action of left multiplication. 
Let $\cF$ 
denote the real vector space of functions \[ f: \cM \times \cQ \times \cV \lra \R((z))^{\hd} \]  satisfying 
the following distribution relation
for each nonzero integer $\lambda$: 
\begin{equation} \label{e:dist}
 f(M, Q, v) = \sgn(\lambda)^n \sum_{\lambda w = v} f(\lambda M, \lambda^{-1} Q, w). 
\end{equation}
  Define a left $\Gamma$-action on $\cF$ as follows.
Given $\gamma \in \Gamma$, choose a nonzero scalar multiple $A = \lambda \gamma$ with $\lambda \in \Z$ such that $A \in M_n(\Z)$.  
For $f \in \cF$, define
\begin{equation} \label{e:vaction}
 (\gamma f)(M, Q, v) =  \sum_{r \in \Z^n / A \Z^n} \sgn(\det A)f(A^t M, A^{-1} Q , A^{-1}(r + v)). 
 \end{equation}
  The distribution relation (\ref{e:dist}) implies that (\ref{e:vaction}) does not depend on the auxiliary choice of $\lambda$.  Note that the action of $\Gamma$ on $\cF$ factors through $\PGL_n(\Q)$. The Solomon--Hu pairing  satisfies  the identity \[ h(\gamma C, v)(zM^t)=\gamma h(C,v)(zM^t)\] for any rational cone $C.$

We can use $\Phi_{\mathrm{Sh}}$ to define a cocycle $\Psi_{\mathrm{Sh}} \in Z^{n-1}(\Gamma, \cF)$ by
\begin{equation} \label{e:psisdef}
 \Psi_{\mathrm{Sh}}(A, M, Q, v) :=   h(\Phi_{\mathrm{Sh}}(A)(Q), v)(z M^t).
\end{equation}
Here and in the sequel we simply write $\Psi_{\mathrm{Sh}}(A, M, Q, v)$ for $\Psi_{\mathrm{Sh}}(A_1, \dotsc, A_n)(M, Q, v)$ with
$A = (A_1, \dotsc, A_n) \in \Gamma^n$.
Our cocycle $\Psi_{\mathrm{Sh}}$ satisfies the following rationality result.
\begin{theorem} \label{t:rational} 
 The value $\Delta^{(k)}  \Psi_{\mathrm{Sh}}(A, M, Q, v)$ lies in the
field $K$ generated over $\Q$ by the coefficients of the polynomial $f_M(x)$.
\end{theorem}

\begin{proof}   We will show that $\Delta^{(k)}(h(C, v)(zM^t))$ lies in $K$ for any rational cone $C$.
By the definition of $f_M(x)$, any automorphism of $\C$ fixing $f_M(x)$ permutes the columns of $M$ up to scaling each
column by a factor $\lambda_i$ such that $\prod_{i=1}^{n} \lambda_i = 1.$  Therefore it suffices to prove that
our value is invariant under each of these operations, namely permuting the columns or scaling the columns by factors whose product is 1.
Now, in the tuple $zM^t$, permuting the columns $M$ 
has the same effect as permuting the variables $z_i$; and scaling the $i$th column of $M$ by 
$\lambda_i$ has the same effect as scaling  $z_i$ by $\lambda_i$.  The desired result then follows from
Lemma~\ref{L:indepofD}.
\end{proof}

\subsection{Special values of zeta functions} \label{s:svz}

Let $F$ be a totally real field, and let $\fa$ and $\ff$ be relatively prime integral ideals of $F$.
 The goal of the remainder of this section is to express the special values  $\zeta_{\ff}(\fa, -k)$ for integers $k \ge 0$ in terms of the cocycle $\Psi_{\mathrm{Sh}}$.  We invoke the notation of Section~\ref{rs:specialize}; in particular we fix an embedding $J: F \hookrightarrow \R^n$.
 
 \bigskip

Let $\cR = \Z[\cM/\cD \times \cQ \times \cV]$ denote the free abelian group on the set $\cM/\cD \times \cQ \times \cV$, which is naturally endowed 
with a left $\Gamma$-action by the action on the sets $\cM/\cD, \cQ,$ and $\cV$.
There is a cycle 
$\fZ_{\ff}(\fa) \in H_{n-1}(\Gamma,\cR)$ associated to our totally real field $F$ and integral ideals $\fa, \ff$.
The cycle consists of the data of  elements $\cA \in \Z[\Gamma^n],$ $M \in \cM/\cD$, $Q \in \cQ$, and $v \in \cV$, defined as
follows.

Fix a $\Z$-module basis $w = (w_1, \dotsc, w_n)$ for $\fa^{-1}\ff$. 
 Let $\{\epsilon_1, \dotsc,  \epsilon_{n-1}\}$ denote a basis of the group $U$ of totally positive units of $F$ congruent to 1 modulo 
$\ff$.  Following (\ref{e:alphadef}), define
\begin{equation} \label{e:Adef2}
 \cA(\epsilon_1, \dotsc, \epsilon_{n-1}) = (-1)^{n-1} w_\epsilon \sum_{\sigma \in S_{n-1}} \sign(\sigma) [(\rho_w(f_{1, \sigma}), \dotsc, \rho_w(f_{n, \sigma}))] \in \Z[\Gamma^n]. 
 \end{equation}
Here $\rho_w$ is the right regular representation of $U$ on $w$ defined in (\ref{e:rhowdef}), and $w_\epsilon$ is the orientation associated
to $J(\epsilon)$ as in (\ref{e:wedef}).

Let $M \in \cM/\cD$ be represented by the matrix \begin{equation} \label{e:mdef}
 \N(\fa)^{1/n} (J_j(w_i))_{i,j=1}^{n} =  \N(\fa)^{1/n} J(w)^t. \end{equation}
Note that $f_M \in \Q[x_1, \dotsc, x_n]$ is the homogeneous polynomial of degree $n$ given by the norm:
\begin{equation} \label{e:Pdef2}
f_M(x_1, \dotsc, x_n) = \N(\fa) \cdot\N(w_1x_1 + \cdots + w_nx_n).
 \end{equation}
Let $Q$ be the image under the embedding $J_n: F \hookrightarrow \R$ of the dual basis to $w$ under the trace pairing on $F$, 
as in (\ref{e:qdef}) :
\begin{equation} \label{e:Qdef2}
 Q = (0, \dotsc, 0, 1)J(w)^{-t} = (J_n(w_1^*), \dotsc,  J_n(w_n^*)),  
 \end{equation}
where $\Tr(w_iw_j^*) = \delta_{ij}$.
Define the column vector \begin{equation} \label{e:vdef2} v = (\Tr(w_1^*), \dotsc, \Tr(w_n^*)), \quad \text{ so that } \quad
1 = v_1 w_1 + v_2w_2 + \cdots + v_nw_n.\end{equation}
Dot product with $(w_1, \dotsc, w_n)$ provides a bijection
$ v + \Z^{n} \longleftrightarrow 1 + \fa^{-1}\ff.$
 
 We now define $\fZ_{\ff}(\fa) \in H_{n-1}(\Gamma, \cR)$ 
 to be the homology class represented by the homogeneous $(n-1)$-cycle
 \[ \tilde{\fZ} = \cA \otimes [(M, Q, v)] \in \Z[\Gamma^{n}] \otimes \cR. \]
The fact that $\tilde{\fZ}$ is a cycle follows from  \cite[Lemma 5]{sczech} as in (\ref{e:alphadef}) 
using the fact that the elements $M$, $Q$, and $v$ are invariant under the action of $\rho_w(U)$.

For each integer $k \ge 0$, the canonical $\Gamma$-invariant map $\cF \otimes \cR \rightarrow \R$ given by
$f \otimes [(M, Q, v)] \mapsto \Delta^{(k)}f(M, Q, v)$ is well-defined by Lemma~\ref{L:indepofD}, and induces via cap product a pairing
\[ \langle \ , \ \rangle_k: H^{n-1}(\Gamma, \cF) \times H_{n-1}(\Gamma, \cR) \longrightarrow \R. \]
Here $\R$ has the trivial $\Gamma$-action.\symbolfootnote[2]{To make contact with the notation of the introduction, note that for each integer $k$ we obtain a map $\eta_k: \cR \ra \cN^\vee$, i.e.\ a pairing $\cN \times \cR \rightarrow \R$,
by $(\Phi, [(M, Q, v)]) \mapsto \Delta^{(k)} h(\Phi(Q), v)(zM^t)$.  The class denoted $\fZ_k$ in the introduction
is the image of $ \fZ_{\ff}(\fa)$ under the map on homology induced by $\eta_k$.}

\begin{theorem} \label{t:psiszeta} We have
$\zeta_{F, \ff}(\fa, -k) = \langle \Psi_{\mathrm{Sh}}, \fZ_{\ff}(\fa) \rangle_k \in \Q.$
\end{theorem}

The rationality of $\zeta_{F, \ff}(\fa, -k)$ is a celebrated  theorem of Klingen and Siegel (see \cite{oda} for  a nice survey of the  history  of  various investigations  on these special values).

The proof we have outlined here is a cohomological reformulation of Shintani's original argument, with the added
benefit that  our definition of $\fZ_{\ff}(\fa)$ gives an explicit signed fundamental domain.

\begin{proof}  Let $U$ denote the group of totally positive units of $F$ congruent to 1 modulo $\ff$,
and let $D = \sum_i a_i C_i$ denote a signed fundamental domain for the action
of $U$ on the totally positive orthant of $\R^n$ (where as in Section~\ref{rs:specialize}, $u \in U$ acts by
componentwise multiplication with $J(u)$).  Then for $\real(s) \gg 0$,
\begin{align}
 \zeta_{F, \ff}(\fa, s) &= \sum_{\stack{\fb \subset \cO_F}{\fb \sim_{\ff} \fa}} \frac{1}{\N\fb^s} 
 =  \sum_{\{y \in 1 + \fa^{-1}\ff, \ y \gg 0\}/U} \frac{1}{(\N\fa \N y)^s} \qquad (\fa^{-1}\fb = (y)) \nonumber \\
 &=  \sum_{\stack{y \in 1 + \fa^{-1}\ff}{J(y) \in D}}  \frac{1}{(\N\fa \N y)^s}. \label{e:yeq}
\end{align}
Here we use the shorthand $\sum_{J(y) \in D}$ for $\sum_i a_i \sum_{J(y) \in C_i}$.
Now Theorem~\ref{t:ddf} implies that, using the notation of  (\ref{e:tphidef}) and 
 (\ref{e:alphadef}), the function
$\phi_{U}(\alpha(\epsilon_1, \dotsc, \epsilon_{n-1}))
$
is the characteristic function ${\mathbf 1_D}$ of such a signed fundamental domain $D$
 for the action of $U$ on $(\R_{>0})^n$.  
 (Recall from Definition~\ref{d:sfd} that if $D = \sum a_i C_i$ is a signed fundamental domain then
 ${\mathbf 1}_D := \sum a_i {\mathbf 1}_{C_i}$.)
 Therefore (\ref{e:jw}) implies that
\[ \Phi_{\mathrm{Sh}, w_1}(\cA, Q) = {\mathbf 1}_{J(w)^{-1} D}. \]

Note that for an element $x \in F$, the vector $v = J(w)^{-1} J(x) \in \Q^n$ satisfies $x = w \cdot v$, where $w = (w_1, \dotsc, w_n)$.
Therefore $J(w)^{-1}D$ consists of rational cones and
\begin{align}
  \langle \Psi_{\mathrm{Sh}}, \fZ_{\ff}(\fa) \rangle_k 
 &=  \Delta^{(k)} \Psi_{\mathrm{Sh}}(\cA, M, Q, v) \nonumber \\
 & = \Delta^{(k)} h(\Phi_{\mathrm{Sh}, w_1}(\cA, Q), v)(zM^t) \nonumber \\
 &= \Delta^{(k)} h({\mathbf 1}_{J(w)^{-1}D}, v)(zM^t) \nonumber \\ 
&= \zeta(J(w)^{-1}D, M, v,- k)  \label{e:zjw}
\end{align}
by Theorem~\ref{t:shintani}.  Here $\Phi_{\mathrm{Sh}, w_1}$ was defined in (\ref{e:phishx}), and may be substituted for
$\Phi_{\Sh}$ since it represents the same cohomology class.
Note also that~\eqref{E:Mpos} is satisfied for each pair $(J(w)^{-1}C_i,M)$ by the definition of $M$ in (\ref{e:mdef}) and the fact that $C_i \subset (\R_{>0})^n$ (which in turn was explained in Remark \ref{r:posorthant}).
 By definition, we have for $\real(s)$ large enough:
\begin{align}
 \zeta(J(w)^{-1}D, M, v, s) &= \sum_{x \in J(w)^{-1}D \cap v + \Z^n} \frac{1}{f_M(x)^s}  \nonumber \\
 &= \sum_{\stack{y \in 1 + \fa^{-1}\ff}{J(y) \in D}} \frac{1}{(\N\fa \N y)^s}, \label{e:jwy}
 \end{align}
 where the last equation uses the substitution $y = w \cdot x$ and (\ref{e:Pdef2}). 
 Comparing (\ref{e:yeq}), (\ref{e:zjw}), and (\ref{e:jwy}) yields the desired equality
 $\zeta_{F, \ff}(\fa, -k) =  \langle \Psi_{\mathrm{Sh}}, \fZ_{\ff, \fa} \rangle_k$.
	
	Finally, the rationality of $\langle \Psi_{\mathrm{Sh}}, \fZ_{\ff}(\fa) \rangle_k =  \Delta^{(k)} \Psi_{\mathrm{Sh}}(\cA, M, Q, v)$
	follows from Theorem~\ref{t:rational}, since $f_M(x)$ has rational coefficients.
\end{proof}

\section{Comparison with the Sczech Cocycle}  \label{s:compare}

In this section we prove that the Shintani cocycle $\Psi_{\mathrm{Sh}}$ defined 
in Section~\ref{s:pscoc} is cohomologous to the one defined by Sczech in \cite{sczech}.
We begin by recalling
the definition of Sczech's cocycle.  The reader is referred to \cite{pcsd}
or \cite{sczech} for a lengthier discussion of Sczech's construction.

\subsection{The Sczech cocycle}

For $n$ vectors $\tau_1, \dotsc, \tau_n \in \C^n$, define a rational function of a variable $x \in \C^n$ by
\begin{equation}\label{E:sczechf} f(\tau_1, \dotsc, \tau_n)(x) = \frac{\det(\tau_1, \dotsc, \tau_n)}{ \langle x, \tau_1 \rangle \cdots \langle x, \tau_n \rangle}.
\end{equation}
The function $f$ satisfies the cocycle relation (see \cite[Lemma 1, pg.\ 586]{sczech})
\begin{equation} \label{e:fcoc}
 \sum_{i=0}^{n} (-1)^i f(\tau_0, \dotsc, \hat{\tau}_i, \dotsc, \tau_n) = 0. 
 \end{equation}
Consider $A = (A_1, \dotsc, A_n) \in \Gamma^n$ and $x \in \Z^n - \{0\}$.
For $i = 1, \dotsc, n$, let $\varpi_i = \varpi_i(A, x)$ denote the  leftmost column of $A_i$ that is not orthogonal
to $x$.  Let $v \in \cV = \Q^n/\Z^n$.  Sczech considers the sum
\begin{equation} \label{e:sczsum}
 \sum_{x \in \Z^n - \{0\}} e(\langle x, v \rangle) f(\varpi_1, \dotsc, \varpi_n)(x), \end{equation}
where $e(u) := e^{2 \pi i u}$.  Although the definition of  $\varpi_i$ ensures that each summand in (\ref{e:sczsum})
is well-defined, the sum itself is not absolutely convergent.  To specify a method of summation,  Sczech 
introduces a vector $Q \in \cQ$ and defines the $Q$-summation
\begin{align}
 \tilde{\Psi}_{\mathrm{Z}}(A, Q, v) &=  (2 \pi i)^{-n} \!\!\!\!\sum_{x \in \Z^n - \{0\}}\!\!\! e(\langle x, v \rangle) f(\varpi_1, \dotsc, \varpi_n)(x)|_Q  \nonumber \\
 &:= (2 \pi i)^{-n} \lim_{t \rightarrow \infty} \sum_{\stack{x \in \Z^n - \{0\}}{|Q(x)| < t}} \!\!\! e(\langle x, v \rangle) f(\varpi_1, \dotsc, \varpi_n)(x). 
 \label{e:sczsumq}
\end{align}
Here the vector $Q$ gives rise to the function $Q(x) = \langle x, Q \rangle$, and the summation over the
region $|Q(x) | < t$ is absolutely convergent for each $t$.

More generally, given a homogeneous polynomial $P \in \C[x_1, \dotsc, x_n]$, Sczech defines
\begin{equation} \label{e:sczgen}  \tilde{\Psi}_{\mathrm{Z}}(A, P, Q, v) = (2 \pi i)^{-n - \deg P} \!\!\!\!\!\!\sum_{x \in \Z^n - \{0\}} \!\!\!\!\! e(\langle x, v \rangle) 
P( - \partial_{x_1}, - \partial_{x_2}, \dotsc, - \partial_{x_n})(f(\varpi_1, \dotsc, \varpi_n))(x)|_Q. 
\end{equation}
Sczech shows that the function $\tilde{\Psi}_{\mathrm{Z}}$ is a cocycle on $\Gamma$ valued in the module $\tilde{\cF}$ defined in Section~\ref{s:gencoc} below.
  In order to make a comparison with our Shintani cocycle $\Psi_{\mathrm{Sh}} \in Z^{n-1}(\Gamma, \cF)$, however, we 
consider now an associated cocycle valued in the module $\cF$ defined in Section~\ref{s:pscoc}.
We prove in Proposition~\ref{p:sczbeta} below that there exists a 
power-series valued cocycle $\Psi_{\mathrm{Z}} \in Z^{n-1}(\Gamma, \cF)$ such that for each integer $k \ge 0$, we have
\begin{equation} \label{e:sczboth}
 \Delta^{(k)} \Psi_{\mathrm{Z}}(A, M, Q, v) = \tilde{\Psi}_{\mathrm{Z}}(A, f_M^k, Q, v). 
 \end{equation}
Our main theorem in this section is:
\begin{theorem} \label{t:compare}  Define $\Psi_{\mathrm{Sh}}^+ \in Z^{n-1}(\Gamma, \cF)$ by
  \[ \Psi_{\mathrm{Sh}}^+(A, M, Q, v) = \frac{1}{2}(\Psi_{\mathrm{Sh}}(A, M, Q, v) + \Psi_{\mathrm{Sh}}(A, M, -Q, v)) \]
and let $\Psi_{\mathrm{P}} \in Z^{n-1}(\Gamma, \cF)$ be the ``polar cocycle" defined by
\begin{equation}\label{defPsipolar}\Psi_{\mathrm{P}}(A, M, Q, v) = \frac{(-1) ^{n+1}\det(\sigma)}{\prod_{j=1}^n zM^t \sigma_j},\end{equation}
where $\sigma=(\sigma_1, \ldots, \sigma_n)$ is the collection of the  leftmost columns of the tuple $A\in \Gamma^n.$
  Then we have the following equality of classes  in $H^{n-1}(\Gamma, \cF)$ : 
  \begin{equation}[\Psi_{\mathrm{Z}}] = [\Psi_{\mathrm{Sh}}^+] + [\Psi_{\mathrm{P}}].\end{equation}
\end{theorem}

\begin{remark} It is proven in \cite[Theorem 3]{sczech} that the cohomology class $[\Psi_{\mathrm{P}}]$ is nontrivial.  However, it clearly vanishes under application of the
Shintani operator $\Delta^{(k)}$ and therefore does not intervene in arithmetic applications.
\end{remark}

\begin{remark} In \cite{sczech}, Sczech considers a matrix of $m$ vectors $Q_i \in \cQ$
and the $Q$ summation (\ref{e:sczgen}) with $Q(x) = \prod Q_i(x)$.
 However, the resulting cocycle is simply the average of the individual cocycles obtained from each $Q_i$.  (This is not clear from the original definition, but follows from Sczech's explicit formulas for his cocycle.)  Therefore it is sufficient to consider just one vector $Q$.
\end{remark}

\begin{remark} In view of Theorem~\ref{t:compare} and (\ref{e:sczboth}), the evaluation
of partial zeta functions of totally real fields using Sczech's cocycle given in \cite[Theorem 1]{sczech}
follows also from our Theorem~\ref{t:psiszeta}.  In fact, we obtain a slightly stronger result in that
we obtain the evaluation using each individual vector $Q_i = J_i(w^*)$, whereas Sczech obtains the result
using the matrix of all $n$ such vectors; it would be interesting to prove this stronger result directly 
from the definition of Sczech's cocycle via $Q$-summation, rather than passing through
the Shintani cocycle and Theorem~\ref{t:compare}.
\end{remark}

\subsection{A generalization of Sczech's construction} \label{s:genscz}

Let $k$ be a positive integer, and let $A = (A_1, \dotsc, A_k) \in \Gamma^k$.
For each tuple $w \in \{1, \dotsc, n\}^k$, let $ B(A, w) \subset \Z^n - \{0\}$ denote the set of vectors
$x$ such that the leftmost column of $A_i$ not orthogonal to $x$ is the $w_i$th, for $i = 1, \dotsc, k$.  In other words,
\[ B(A, w) = \bigcap_{i=1}^{k} \{x \in \Z^n: \langle x, A_{i j} \rangle = 0 \text{ for } j < w_i, 
 \langle x, A_{i w_i} \rangle \neq 0 \}. \]
 Here $A_{ij}$ denotes the $j$th column of the matrix $A_i$.  Then 
 \[ \Z^n - \{0\} = \bigsqcup_{w \in \{1, \dotsc, n\}^k} B(A, w). \]
Sczech's sum (\ref{e:sczsumq}) can be written with $k=n$ as:
\[ \tilde{\Psi}_{\mathrm{Z}}(A, Q, v) = \sum_w \sum_{x \in B(A,w)} e(\langle x, v \rangle) f(A_{1 w_1}, \dotsc, A_{n w_n})(x)|_Q. 
\]
We now generalize this expression by replacing  the columns $A_{i w_i}$ with certain other columns of the matrices $A_i$.

Write $S_k = \{1, \dotsc, k\}$ and for simplicity let $S = S_n$.
 Given $A = (A_1, \dotsc, A_k) \in \Gamma^k$ and
 an element $t = ((a_1, b_1), (a_2, b_2), \dotsc, (a_n, b_n)) \in (S_k \times S)^n$, 
define \[ \tau(A, t) = (A_{a_1 b_1}, A_{a_2 b_2}, \dotsc, A_{a_n b_n}). \]
In other words, $\tau(A, t)$ is an $n \times n$ matrix whose $i$th column is the $b_i$th column of $A_{a_i}$.

For any function $g: S^k \rightarrow (S_k \times S)^n$, we would like to consider the sum
\begin{equation} \label{e:psigdef}
 \psi(g)(A, Q, v) = \sum_w \sum_{x \in B(A,w)} e(\langle x, v \rangle) f(\tau(A, g(w)))(x)|_Q. 
\end{equation}
For example, $\tilde{\Psi}_{\mathrm{Z}} = \psi(\beta)$ where $\beta(w) = ((1, w_1), (2, w_2), \dotsc, (n, w_n)).$  The difficulty 
with (\ref{e:psigdef}) in general, however, is that the denominators in the expression defining $f$ may vanish; it is
therefore necessary to introduce an auxiliary variable $u \in \C^n$ and to consider the function 
\begin{equation} \label{e:psigdef2}
 \psi(g)(A, Q, v, u) = \sum_w \sum_{x \in B(A, w)} e(\langle x, v \rangle) f(\tau(A, g(w)))(x-u)|_Q. 
\end{equation}
By Sczech's analysis~\cite[Theorem 2]{sczech}, this $Q$-summation converges for all $u\in\C^n$ such that the map $x\mapsto f(\tau(A,g(w)))(x-u)$ is defined on $B(A,w)$, i.e., such that the denominator of the right hand side of~\eqref{E:sczechf} is nonzero. Thus it converges for $u$ in a dense open subset of $\C^n$ that consists of the complement of a countable union of hyperplanes. In fact, this convergence is uniform for $u$ in sufficiently small compact subsets of $\C^n$.

This formalism allows for the construction of homogeneous cochains in $C^{k-1}(\Gamma, \cF)$
as follows.
\begin{prop} \label{p:psigdef}
For any function $g: S^k \rightarrow (S_k \times S)^n$ and $A = (A_1, \dots A_k) \in \Gamma^k$,
 there is a unique power series
\[ \Psi(g)(A, Q, v) \in \Q((z)) \] such that
\begin{equation} \label{e:psieq}
 \psi(g)(A, Q, v, u) = (2\pi i)^{n}\Psi(g)(A, Q, v)( 2\pi i u)
\end{equation}
for any $u \in \C^n$ for which  (\ref{e:psigdef2}) is defined.
Furthermore, for any $M \in \cM$ we have \[ \Psi(g)(A, Q, v)(zM^t) \in \R((z))^{\hd},\]
and the assignment $(A, M, Q, v) \mapsto \Psi(g)(A, Q, v)(zM^t)$ is a homogeneous
cochain in $C^{k-1}(\Gamma, \cF)$.
\end{prop}

The following lemma is the technical heart of the proof of Proposition~\ref{p:psigdef} and is proven by
reducing to computations in \cite{sczech}.

\begin{lemma}\label{l:psidef}
Let $H\subset \Q^n$ be a vector subspace and let $L=H\cap \Z^n$.  Let $\tau=(\tau_1,\ldots,\tau_n)\in M_n(\Z) \cap \Gamma$.  Then for every $v\in\Q^n$,
\[
G(u):=\sum_{x\in L}e(\langle x,v\rangle)f(\tau_1,\ldots,\tau_n)(x-u)|_Q \]
belongs to $(2\pi i)^{n}\Q((2\pi i u))$.  If $M=(m_{ij})\in \cM$, then $G(uM^t)\in (2\pi i)^{n}\Q(\{m_{ij}\})((2\pi i u))^{\hd}$.
\end{lemma}
\begin{remark}
As with~\eqref{e:psigdef2}, the $Q$-summation defining $G(u)$ converges for $u$ in a dense open subset of $\C^n$ that consists of the complement of a countable union of hyperplanes.  The convergence is uniform for $u$ in sufficiently small compact sets.
\end{remark}
\begin{proof}
Set $x'=x\tau$, $u'=u\tau$, and $Q'=\tau^{-1}Q$.  Then
\begin{align*}
G(u) &= \sum_{x\in L}\frac{e(\langle x,v\rangle)\det(\tau)}{\langle x-u,\tau_1\rangle\cdots \langle x-u, \tau_n\rangle}\Big|_Q \\
&=
\sum_{x'\in L\tau}\frac{e(\langle x',\tau^{-1}v\rangle)\det(\tau)}{(x_1'-u_1')(x_n'-u_n')}\Big|_{Q'}.
\end{align*}
Suppose first that $H=\Q^n$, so that $L=\Z^n$.  Then $L\tau$ is a finite-index sublattice of $L$. Since the nontrivial characters of $L/L\tau$ are $x\mapsto e(\langle x,\tau^{-1}y\rangle)$ for $y\in L^*/\tau L^*$, we have the Fourier expansion
\[
\mathbf{1}_{L\tau}(x') = \frac{1}{|\det\tau|}\sum_{y\in L^*/ \tau L^*}e(\langle x',\tau^{-1}y\rangle).
\]
($L^*$ is the dual lattice of $L$, with its elements naturally viewed as column vectors.)  
Therefore,
\[
G(u)=s_\tau \sum_{y\in L^*/\tau L^*}\sum_{x'\in L}
\frac{e(\langle x',\tau^{-1}(v+y)\rangle)}
{(x_1'-u_1')\cdots(x_n'-u_n')}\Big|_{Q'},
\] where $s_\tau = \sgn(\det \tau).$
Letting $p=u'-x'$, we obtain
\[
G(u) =s_\tau \sum_{y\in L^*/\tau L^*}\mathscr{C}_1(u',\tau^{-1}(v+y),Q')
\]
where, adopting notation from \cite[(3)]{sczech},
\[
\mathscr{C}_1(u,v,Q)=\sum_{p\in \Z^n\tau+u}\frac{e(\langle u-p,v\rangle)}{p_1\cdots p_n}\Big|_Q.
\]
The fact that $G(u)$
belongs to $(2\pi i)^{n}\Q((2\pi i u))$ now follows from Sczech's evaluation of $\mathscr{C}_1(u,v,Q)$ in elementary terms
given in~\cite[Theorem 2]{sczech}.

Now suppose $r:=\dim H<n$.  Choose a matrix $\lambda=(\lambda_1,\ldots,\lambda_r)\in M_{n\times r}(\Z)$ whose column space is $H^\perp$.  Then $L\tau=(H\cap \Z^n)\tau$ has finite index in 
\[
K:= H\tau\cap \Z^n=\{x\in\Z^n : \langle x,\lambda_1\rangle =\cdots \langle x,\lambda_r\rangle=0\}.
\]
Inserting the character relations as above, we have
\[
G(u)= s_\tau \sum_{y\in K^*/\tau L^*}\sum_{x'\in K}
\frac{e(\langle x',\tau^{-1}(v+y)\rangle)}
{(x_1'-u_1')\cdots(x_n'-u_n')}\Big|_{Q'}.
\]
Computing as in~\cite[page 599]{sczech}, we obtain 
\begin{align*}
G(u)
&= s_\tau  \sum_{y\in K^*/\tau L^*} \mathscr{C}_1(\lambda,u',\tau^{-1}(v+y),Q'), 
\end{align*}
where
\[
\mathscr{C}_1(\lambda,u,v,Q)=\int_0^1\cdots\int_0^1
\mathscr{C}_1(u,t_1\lambda_1+\cdots t_r\lambda_r+v,Q)dt_1\cdots dt_r.
\]
By the proof of~\cite[Lemma 7]{sczech}, we have
\[
\mathscr{C}_1(\lambda,u',\tau^{-1}(v+y),Q')\prod_{i=1}^n (1-e(u_i'))\in  (2\pi i)^{n}\Q[[2\pi i u']]=(2\pi i)^{n}\Q[[2\pi i u]].
\]
for all $y$.  Thus $G(u)\in (2\pi i)^{n}\Q((2\pi i u))$.

To prove the last statement of the lemma, let  $M=(m_{ij})\in \cM$.  Then $e(uM^t\tau_j)-1= 2\pi i(uM^t\tau_j)H(u)$ for an invertible power series  $H(u) \in \Q(\{m_{ij}\})[[2\pi i u]]$.  Since $M \in \cM$, no component of $M^t\tau_j$ is equal to zero.  Therefore, $uM^t\tau_j\in\Q(\{m_{ij}\})[u]$ is a dense linear form  (see the paragraph following Lemma~\ref{l:ppas} for the terminology). The desired result then follows from the proof of Lemma~\ref{L:hd}.
\end{proof}

\begin{proof}[Proof of Proposition~\ref{p:psigdef}]	
Each $B(A,w)$ has the form $L-\bigcup_i M_i$ where $L$ is a sublattice of $\Z^n$ and the $M_i$ are finitely many distinct sublattices of $L$ with positive codimension.  The existence of the functions $\Psi(g,A,Q,v)$ now follows from inclusion-exclusion and Lemma~\ref{l:psidef}, as does the fact that $\Psi(g)(A,Q,v)(zM^t)$ belongs to $\R((z))^{\hd}$ when $M\in \cM$.

To see that $(A,M,Q,v)\mapsto\Psi(g)(A,Q,v)(zM^t)$ is homogeneous $(k-1)$-cochain, we first observe that for any $C\in\Gamma$, 
we have $\tau(CA,g(w))=C\tau(A,g(w))$.  
Furthermore, it is easy to see that $B(CA,w)=B(A,w)C^{-1}$, and a straightforward  change of variables then shows that
\begin{align*}
\psi(g)(CA,Q,v,uM^t) &=
\psi(g)(A,C^{-1}Q,C^{-1}v,uM^tC).\qedhere
\end{align*}
\end{proof}

\subsection{Recovering the Sczech and Shintani cocycles}

We apply the formalism of the previous section
to the following functions $S^n \rightarrow (S \times S)^n$:
\begin{align*}
\alpha(w) &= ((1, 1), (2,1), \dotsc, (n,1)), \\
\beta(w) &= ((1, w_1), (2, w_2), \dotsc, (n, w_n)).
\end{align*}
As we now show, the power series $\Psi(\alpha)$ and $\Psi(\beta)$ associated to these functions are the
Shintani and Sczech cocycles, respectively (up to an error term given by the polar cocycle in the first instance).

First we prove a lemma that evaluates the Shintani operator on a regular power series twisted by
 $M \in \cM$.  Let $f_M$ be the polynomial  defined in (\ref{e:fmdef}). Let $\sigma \in M_n(\Z)$ and 
 define coefficients $P_r^k(\sigma)$ indexed by tuples $r = (r_1, \dotsc, r_n)$ of nonnegative integers
by the formula
\begin{equation} \label{e:prdef}
 f_M(z \sigma^t)^k = \sum_{r} \frac{P_r^k(\sigma)}{r!} z_1^{r_1} \cdots z_n^{r_n},
\end{equation}
where $r! := r_1! \cdots r_n!.$  When  $\sigma = 1,$ we simply write  $P_r^k = P_r^k(1).$ 

\begin{lemma} \label{l:Mtwist}
Let \[ F(z) = \sum_r F_r z^r
\in K[[z_1, \dots, z_n]], \] where $r$ ranges over $n$-tuples of nonnegative integers. 
Let $M \in \cM$.
 Then
\[ \Delta^{(k)} F(z M^t) = \sum_{r} F_r P_r^k.
\]

\end{lemma}

\begin{proof}  We have \[ \Delta^{(k)}(F(z M^t)) = \sum_{r} F_r \Delta^{(k)}((z M^t)^r). \]
As noted in Remark~\ref{r:shreg},  $\Delta^{(k)}$ evaluated on a regular power series equals $(k!)^n$ times the coefficient
of $z_1^k \cdots z_n^k$.
Meanwhile $P_r^k$ is $r!$ times the coefficient of $z^r$ in $(zM)_1^k \cdots (zM)_n^k$.
The desired result then follows (with $s = (k, k, \dotsc, k)$) from the following general reciprocity
law for any tuples $r$ and $s$ such that $\sum r = \sum s = m$.   If we let 
\[ C_{r,s}(M) = s! \cdot (\text{coefficient of } z^s \text{ in } (zM)^{r}), \]
  then \begin{equation} \label{e:crs}
  C_{r,s}(M) = C_{s,r}(M^t). \end{equation}  To see this, note that
  \[ C_{r,s}(M) = \frac{1}{m!}\left. \left(\frac{\partial}{\partial z_1}\right)^{s_1} \cdots  \left(\frac{\partial}{\partial z_n}\right)^{s_n}
   \left(\frac{\partial}{\partial y_1}\right)^{r_1} \cdots  \left(\frac{\partial}{\partial y_n}\right)^{r_n}(z M y)^m \right|_{z = y = (0, \dots, 0)}.
  \]
  This expression is clearly invariant upon switching $r$ and $s$ and replacing $M$ by $M^t$.
\end{proof}

\begin{prop} \label{p:sczbeta}
Let $\beta(w) = ((1, w_1), (2, w_2), \dotsc, (n, w_n)).$  Define
\begin{equation} \label{e:psiZpsibeta}
 \Psi_{\mathrm{Z}}(A, M, Q, v) = \Psi(\beta)(A, Q, v)(zM^t). \end{equation}
Then $\Psi_{\mathrm{Z}}$ satisfies (\ref{e:sczboth}).
\end{prop}

One can show directly from the definition (\ref{e:psiZpsibeta}) that $\Psi_{\mathrm{Z}} \in Z^{n-1}(\Gamma, \cF)$, but this follows
also from our proof of Theorem~\ref{t:compare} so we omit the details.

\begin{proof}  By the definition of $\beta$, the function $\psi(\beta)(A, Q, v, u)$ is well-defined for all $u$ in an open
neighborhood of $0$ in  $\C^n$.
 Therefore $F = \Psi(\beta)(A, Q, v)$ is a regular power series, i.e.\  $F(z) \in \R[[z_1, \dotsc, z_n]]$.
Hence if we write $F = \sum_r F_r z^r$ and $f_M^k(z) = \sum_r P_r^k z^r/r!$ as in (\ref{e:prdef}), then Lemma~\ref{l:Mtwist}
implies that 
\begin{equation} \label{e:beta1}
\Delta^{(k)} F(zM^t) = \sum_r F_r P_r^k. 
\end{equation} 
On the other hand, by \cite[Theorem 2]{sczech} the series 
\[
 F(u) = (2 \pi i )^{-n} \sum_{x \in \Z^n - \{0\}} e(\langle x, v \rangle) f(\varpi_1, \dotsc, \varpi_n)\left(x - \frac{u}{2 \pi i}\right)|_Q,
 \]
 as well as  those  formed by taking partial derivatives of the general term,  converge uniformly on a sufficiently small compact  neighborhood of $u=0$ in $ \C^n.$ 
Therefore  term by term differentiation is valid for $F$, and  after  applying $f_M^k( \partial_{u_1}, \dotsc, \partial_{u_n})$ and plugging in $u = 0$ we obtain
 \begin{equation} \label{e:beta2}
 \sum_r F_r P_r^k = (2 \pi i)^{-n(k+1)} \!\!\!\!\!\sum_{x \in \Z^n - \{0\}} e(\langle x, v \rangle) f_M^k( - \partial_{x_1}, - \partial_{x_2}, \dotsc, - \partial_{x_n})(f(\varpi_1, \dotsc, \varpi_n))(x)|_Q. 
 \end{equation}
The right side of (\ref{e:beta2}) is the definition of $\tilde{\Psi}_{\mathrm{Z}}(A, f_M^k, Q, v)$,
so combining (\ref{e:beta1}) and (\ref{e:beta2}) gives the desired equality
\[ \Delta^{(k)} F(zM^t) =  \tilde{\Psi}_{\mathrm{Z}}(A, f_M^k, Q, v).  \]
\end{proof}

\begin{prop} \label{p:shinalpha}
Let $\alpha(w) = ((1, 1), (2,1), \dotsc, (n,1))$.  Then
\[ \Psi_{\mathrm{Sh}}^+(A, M, Q, v)+\Psi_{\mathrm{P}}(A, M, Q, v) = \Psi(\alpha)(A, Q, v)(zM^t). \]
\end{prop}

\begin{proof}  Attached to $A$ and $\alpha$ is the square matrix $\sigma=(\sigma_1, \ldots, \sigma_n)=(A_{11}, \ldots, A_{n1})$.
Arguing as in the first paragraph of the proof of Lemma~\ref{l:psidef}, we have 
\begin{equation}\label{eqpsialpha}
\psi(\alpha)(A, Q, v, u) = \ s_\sigma  \sum_{y\in \Z^n/\sigma \Z^n}\sum_{x' \in  \Z^n - \{0\}} \frac{e(\langle x', \sigma^{-1}(v+y)\rangle) }{ (x'_1-u'_1)\dotsc (x'_n-u'_n)}\mid_{Q'}.
\end{equation}
 For any  given $y\in \Z^n,$ the inner sum is identified in Sczech's  notation  \cite[(3)]{sczech}   as 
\begin{align}\label{eqC1def}
(-1)^n \sum_{x\in  \Z^n - \{0\}} \frac{e(\langle x, \sigma^{-1}(v+y)\rangle) }{ (x_1-u'_1)\dotsc (x_n-u'_n)}\mid_{Q'}= \frac {-1} {u'_1\dotsc u'_n}+\mathscr C_1 (u', \sigma^{-1}(v+y), Q').
 \end{align}

 To express this last quantity  in elementary terms, we write  $v'=\sigma^{-1}(v+y)$  and let $J=J(y)$  denote the set of indices $j\in \{1, \ldots, n\}$ with $v'_j\in \Z.$ We  then invoke  
  \cite[Theorem 2]{sczech} to obtain for $u\in (\C-\Z)^n$ :
\begin{align}\label{C1Q}
\mathscr C_1 (u, v', Q')=& \ \frac 12 (\mscr{H}(u,v', Q')+(-1)^n \mscr{H}(-u, -v', Q'))\nonumber
\\= &\ \frac 12 (\mathscr{H}(u,v', Q')+ \mscr{H}(u, v', -Q')),\\ 
\textrm{where } \ \mathscr{H}(u,v',Q')= &\   (-2\pi i)^n \prod_{j\in J} \left(\frac{e(u_j)}{1-e(u_j)}+\frac{1+\sign Q'_j} 2\right) \prod_{j\notin J} \frac{e( u_j \{v'_j\})} {1-e(u_j)}.\nonumber
\end{align}
Fix a subset $I_0 \subset J$, select the factor  $\frac{(1+\sign Q'_j)} 2$ for $j\in I_0$, and  expand the product for $\mscr{H}(u,v', Q')$ accordingly. Writing $I=\overline{I_0}$
one obtains 
     \begin{align} 
   \mscr{H} (u,v', Q')= &(-2\pi i )^n \sum_{I_0 \subset J} \prod_{j\in I_0}\frac{(1+\sign Q'_j)} 2 \prod_{j\in J-I_0} \frac{e(u_j)}{1-e(u_j)}  \prod_{j\notin J}  \frac{e(u_j\{v'_j\})}{1-e(u_j)}\nonumber
\\
   =&\  (-2\pi i)^n \sum_{I\supset \overline{J}} \weight(C_I) \frac{e(u\cdot \sigma^{-1}a_I)}{\prod_{j \in I} 1-e(u_j)}.
\label{eqhfinal} \end{align}  
The last line follows from the formula (\ref{e:weight}) for $\weight(C_I)$ and  the definition (\ref{e:aIdef}) of the point $a_I \in\sP_{I}.$ 
Collecting (\ref{eqpsialpha}), (\ref{eqC1def}), (\ref{C1Q}) and (\ref{eqhfinal})  we arrive at 
\begin{equation} \label{idpsialpha}
\psi(\alpha)(A, Q, v, u) = \frac{(-1)^{n+1} \det\sigma}{N (u\sigma)} +  (2\pi i)^n s_\sigma\sum_{\stack{y\in \Z^n/\sigma \Z^n}{I\supset
 \overline{J(y)}}} \textrm{weight}^{\!+}(C_{I}) \frac{e(u\cdot a_I)}{\prod_{j\in I} 1-e((u\sigma)_j)},
\end{equation}
 where   $\textrm{weight}^{\!+}(C_I)$ is the average of the weights for $Q$ and $-Q.$ The identity (\ref{idpsialpha})  holds for all $u\in \C^n$ as long as the vector $u\sigma$ has no component in $\Z.$

Unwinding the argument of Section~\ref{s:bookkeeping} to go from (\ref{e:cqdecomp0}) to (\ref{e:cqdecomp}), we obtain
\begin{align} \label{idinhpsialpha}
\psi(\alpha)(A, Q, v, u) =&\  \frac{(-1)^{n+1} \det\sigma}{N (u\sigma)} +  (2\pi i)^n s_\sigma \!\!\!\!\! \sum_{I\subset \{1, \ldots, n\}} \!\!\!\!\! \textrm{weight}^{\!+}(C_{I}) \!\!\!\!\! \sum_{a\in  \sP_{I}\cap (v+\Z^n) }\frac{e(u\cdot a)}{\prod_{j\in I} 1-e((u\sigma)_j)}\nonumber\\
=&\  \frac{(-1)^{n+1} \det\sigma}{N (u\sigma)} +  (2\pi i)^n s_\sigma \!\!\!\!\! \sum_{I\subset \{1, \ldots, n\}} \!\!\!\!\! \textrm{weight}^{\!+}(C_{I}) h(C_{I},v)(2\pi i u)\nonumber\\
=& \  \frac{(-1)^{n+1} \det\sigma}{N (u\sigma)} +  (2\pi i)^n\,  h(\Phi_{\mathrm{Sh}}^+(A) ( Q) ,v)(2\pi i u),
\end{align}
 where the superscript ``+"  again denotes the average of the contributions of $Q$ and $-Q.$  This implies   the desired equality
  between power series using  the definitions of $\Psi_{\mathrm{Sh}},\Psi_{\mathrm{P}}$, and $\Psi(\alpha)$ given in 
 (\ref{e:psisdef}), (\ref{defPsipolar}), and (\ref{e:psieq}) respectively.
 \end{proof}

\subsection{An explicit coboundary}

Recall the notation $S_k = \{1, \dotsc, k\}$, $S = S_n$.
Extending by linearity, we can define $\Psi(g)$ for any map $g\colon S^k \rightarrow \Z[(S_k \times S)^n]$.  In fact, more is
true; if we denote by $\partial\colon \Z[(S_k \times S)^{n+1}] \rightarrow \Z[(S_k \times S)^{n}]$ the usual differential
\[ \partial([t_0, \dotsc, t_{n}]) = \sum_{i=0}^{n} (-1)^i [t_0, \dotsc, \hat{t}_i, \dotsc, t_n], \]
then $\Psi(g)$ is well-defined for any map $g\colon S^k \rightarrow \Z[(S_k \times S)^n]/\Image(\partial)$.
This follows from the cocycle relation (\ref{e:fcoc}), which implies that $f(\tau(A, t)) = 0$ for
$t \in \Image(\partial).$

We will show that the map $\beta - \alpha\colon S^n \rightarrow \Z[(S \times S)^n]/\Image(\partial)$ is a coboundary
in the following sense.  
 For $i = 1, \dotsc, n$, let $\hat{e}_i: S_{n-1} \rightarrow S$ be the unique increasing map whose image does not contain $i$.  Given 
\begin{equation} \label{e:hdef}
h: S^{n-1} \rightarrow \Z[(S_{n-1} \times S)^n], 
\end{equation}
 define $d h: S^n \rightarrow \Z[(S \times S)^n] $
by
\begin{equation} \label{e:ddef}
 (dh)(w_1, \dotsc, w_n) = \sum_{i=1}^n (-1)^i (\hat{e}_i \times \id)(h(\hat{w}_i)). 
 \end{equation}
We  will show that there exists an $h$ such that $\beta - \alpha = d h \pmod{\Image(\partial)}$.
Let us indicate why this completes the proof of Theorem~\ref{t:compare}.  For any $h$ as in (\ref{e:hdef}),
Proposition~\ref{p:psigdef} yields 
 a homogeneous cochain  $\Psi(h) \in C^{n-2}(\Gamma, \cF)$.  It is easily checked
from (\ref{e:ddef}) that $d(\Psi(h)) = \Psi(dh)$.  Therefore, combining Propositions~\ref{p:sczbeta} and \ref{p:shinalpha},
we obtain 
\[ \Psi_{\mathrm{Z}} - \Psi_{\mathrm{Sh}}^+-\Psi_{\mathrm{P}} = \Psi(\beta) - \Psi(\alpha) = d \Psi(h) \]
as desired.
It remains to define the appropriate function $h$.

\begin{prop} \label{e:defhcoboundray}For $i=1, \dotsc, n-1$, define $h_i\colon S^{n-1} \rightarrow \Z[(S_{n-1} \times S)^n] $
by
\[ h_i(w) = 
\begin{cases}
[(1,w_1), \dotsc, (i-1, w_{i-1}), (i, 1), (i, w_i), (i+1, 1),  \dotsc, (n-1,1)], &  w_i > 1 \\
0, &  w_i = 1,
\end{cases}
\]
where  $w = (w_1, \dotsc, w_{n-1})$.
Let $h = \sum_{i=1}^{n-1} (-1)^i h_i.$  Then $\beta - \alpha \equiv dh \pmod{\Image(\partial)}$.
\end{prop}

\begin{remark}
For $n=2$, the map $h$ is given by  $h(1) = 0$ and $h(2) = -[(1,1), (1,2)]$.
  This is the formula stated by Sczech \cite[Page 371]{sczcom}.
\end{remark}

\begin{proof}
One shows by induction on $m$  that for $m = 1, \dotsc, n$, 
\[  \left(\alpha  + \sum_{i=1}^{m-1} dh_i \right)(w_1, \dotsc, w_n) \]
is equal to 
\begin{align*}
&  [(1, w_1), \dotsc, (m-1, w_{m-1}), (m, 1), \dotsc, (n,1)]  + \\
&  \sum_{i=1}^{m-1} (-1)^{i+m-1} [(1, w_1), \dotsc, \widehat{(i, w_i)}, \dotsc, (m-1, w_{m-1}), (m, 1), 
(m, w_{m}), (m+1, 1), \dotsc, (n,1)].
\end{align*}
For $m=n$, this yields
\begin{align*}
 (\alpha + dh)(w_1, \dotsc, w_n) &=   [(1, w_1), \dotsc, (n-1, w_{n-1}), (n, 1)]  + \\
& \ \ \ \ \  \sum_{i=1}^{n-1} (-1)^{i+n-1} [(1, w_1), \dotsc, \widehat{(i, w_i)}, \dotsc, (n-1, w_{n-1}), (n, 1), 
(n, w_{n})] \\
 &\equiv  [(1,w_1), \dotsc, (n,w_n)]  \pmod{\Image(\partial)} 
 \end{align*}
  as desired.
\end{proof}

\section{Integral Shintani cocycle} \label{s:smooth}

In this section we introduce an auxiliary prime $\ell$ and enact a smoothing process on our cocycle
$\Psi_{\mathrm{Sh}}$ to define a cocycle $\Psi_{\mathrm{Sh},\ell}$ on a certain congruence subgroup of $\Gamma$.
The smoothed cocycle $\Psi_{\mathrm{Sh},\ell}$ satisfies an integrality property
refining the rationality result stated in Theorem~\ref{t:rational}.   This refinement is stated in Theorem~\ref{t:smooth} below.
The key technical result allowing the proof of Theorem~\ref{t:smooth} is the explicit formula for  $\Psi_{\mathrm{Sh},\ell}$
given in Theorem~\ref{t:lexplicit}.
We provide the details of the proof of Theorem~\ref{t:lexplicit} here; the deduction of Theorem~\ref{t:smooth} 
is given in \cite[\S2.7]{pcsd}.

The arithmetic applications regarding classical and $p$-adic $L$-functions of totally real fields stated in the Introduction as Theorems~\ref{t:integral}, \ref{t:padic}, and~\ref{t:oov} follow {\em mutatis mutandis} as in \cite{pcsd} from Theorem~\ref{t:smooth}.  See \S3-5 of {\em loc.~cit.~}for the proofs.

\subsection{Definition of the smoothing}
Fix a prime $\ell$.  Let $\Z_{(\ell)} = \Z[1/p, p \neq \ell]$ denote the localization of $\Z$ at the prime ideal $(\ell)$.
Let
\[ \Gamma_\ell = \Gamma_0(\ell \Z_{(\ell)}) = \{ A \in \GL_n(\Z_{(\ell)}):  \ell \mid A_{j1} \text{ for } j > 1   \}.\]
 Let $\pi_\ell = \diag(\ell, 1, 1, \dotsc, 1).$  Note that if $A \in \Gamma_\ell$, then $\pi_\ell A \pi_\ell^{-1} \in \GL_n(\Z_{(\ell)}).$

  For any $\Psi \in Z^{n-1}(\Gamma, \cF)$, define a smoothed homogeneous cocycle $\Psi_{\ell} \in Z^{n-1}(\Gamma_\ell, \cF)$ by
\begin{equation}
 \label{e:psildef}
 \Psi_\ell(A, M, Q, v) :=  \Psi(\pi_\ell A \pi_\ell^{-1}, \pi_\ell^{-1}M,  \pi_\ell Q , \pi_\ell v) - \ell \Psi(A, M, Q, v)
 \end{equation}
for $A = (A_1, \dots, A_n) \in \Gamma_\ell^n$.  
The following is a straightforward computation using the fact that $\Psi$ is a cocycle for $\Gamma$. 
\begin{prop}  We have $\Psi_\ell \in Z^{n-1}(\Gamma_\ell, \cF)$.
\end{prop}

\subsection{An explicit formula} \label{s:explicit}

We will now give an explicit formula for $\Delta^{(k)} \circ\Psi_{\mathrm{Sh},\ell}$ for an integer $k \ge 0$ in terms of Dedekind sums.
For each integer $k \ge 0$, the Bernoulli polynomial $b_k(x)$ is defined by 
the generating function 
\begin{equation} \label{e:bdef}
 \frac{te^{xt}}{e^t - 1} = \sum_{k=0}^{\infty} b_k(x)\frac{t^k}{k!}. 
 \end{equation}
The following elementary lemma gives an explicit formula for the terms appearing in the definition of $h(C, v)$.
\begin{lemma} \label{l:bern}
 Consider the cone $C = C(\sigma_{i_1}, \dotsc, \sigma_{i_r})$ whose generators are a subset of the columns of the matrix $\sigma \in \Gamma$.
 Then
\begin{equation} \label{e:bern}
 \frac{e^{z\cdot a }}{(1 - e^{z \cdot \sigma_{i_1}})\cdots(1 - e^{z\cdot \sigma_{i_r}})} = (-1)^r \sum_{\stack{m_j=0}{r\text{-tuples }} }^\infty \prod_{j=1}^{r} 
\frac{b_{m_j}(\sigma^{-1}(a)_{i_j})}{m_j!} (z \sigma_{i_j})^{m_j-1} \prod_{i \not\in \{i_j\}} e^{(z \sigma_i)(\sigma^{-1}(a)_i)} . 
\end{equation}
\end{lemma}

Define the periodic Bernoulli function $B_k(x) = b_k(\{x\})$, where
$\{x\} \in [0, 1)$ denotes the fractional part of $x$.
The functions $B_k$ are continuous for $k \neq 1$, i.e.\ $b_k(0) = b_k(1)$.
The function $B_1$ is not continuous at integers since $b_1(0) = -1/2$ and
$b_1(1) = 1/2$.  One can choose between these values by means of an auxiliary $Q \in \cQ$ as follows.

\begin{definition}
Let $e=(e_1,\ldots, e_n)$ be a vector of positive integers,  $Q \in \cQ$, and $v\in \cV.$
Let 
\begin{equation}\label{defJ}
J=\{1\leq j\leq n :  e_j=1  \textrm{ and } v_j\in \Z\}.\end{equation}
Define
$$\B_e(v,Q) = \left(\prod_{j\in J} \frac{-\sgn(Q_{j})} 2 \right)  \prod _{j\notin J}  B_{e_j}(v_j).$$
Note that this is $\B_e(v, -Q)$ in the notation of \cite{pcsd}.
\end{definition}

Let $\sigma \in M_n(\Z)$ have nonzero determinant.  Define the Dedekind sum
\begin{equation} \label{e:dedsum}
  \D(\sigma, e, Q, v) = 
\sum_{x\in  \Z^n/\sigma \Z^n }  \B_e(\sigma^{-1}(x+v), \sigma^{-1}Q). \end{equation}
Suppose that $\sigma$ has the property that $\sigma_\ell := \pi_\ell \sigma/\ell \in M_n(\Z)$. (This says that the bottom $n-1$ rows of
$\sigma$ are divisible by $\ell$.)  Write $\underline{e} = \sum e_i$.
Define the $\l$-smoothed Dedekind sum
\begin{equation}
\D_\l(\sigma, e,Q, v) = \D(\sigma_\ell, e,  \pi_\ell Q, \pi_\ell v)-\l^{1-n+\underline{e}} \D(\sigma,e,  Q, v) 
\label{e:dldef}
\end{equation}

We can now give a formula for $\Delta^{(k)} \circ\Psi_{\mathrm{Sh},\ell}$ in terms of the smoothed Dedekind sum $\D_\ell$.
Let $A = (A_1, \dotsc, A_n) \in \Gamma_\ell^n$, and let $\tilde{\sigma}$ denote the matrix consisting of the
first columns of the $A_i$.  
Assume that $\det \tilde{\sigma}\neq 0,$ and choose a scalar multiple $\sigma = \lambda \tilde{\sigma}$
with $\lambda$ an integer coprime to $\ell$ such that $\sigma \in M_n(\Z)$.  Note that since each $A_i \in \Gamma_\ell$,
it follows that $\sigma_\ell = \pi_\ell \sigma / \ell \in M_n(\Z)$ as well.

\begin{theorem} \label{t:lexplicit} We have
\[ \Delta^{(k)}\Psi_{\mathrm{Sh},\ell}(A, M, Q, v) = (-1)^n \sgn( \det \sigma)\sum_{r} \frac{P_r^k(\sigma)}{\ell^{\underline{r}}(r+1)!} \D_\ell(\sigma, r + 1, Q, v), \]
where $r+1 := (r_1 + 1, \dotsc, r_n + 1),$
and the coefficients $P_r^k(\sigma)$ are defined in (\ref{e:prdef}).
\end{theorem}

The proof of Theorem~\ref{t:lexplicit} is involved and technical; the reader is invited to move on to the statement of Theorem~\ref{t:smooth}
and the rest of the paper, returning to our discussion here as necessary.

\medskip

 The proof of Theorem~\ref{t:lexplicit} will be broken into three parts: 
\begin{itemize}
\item Showing that the terms
from (\ref{e:bern}) arising from indices $m_j =0$ cancel under the smoothing operation; in particular, $\Psi_{\mathrm{Sh},\ell}$ takes values
in $\R[[z_1, \dotsc, z_n]].$
\item Calculating the remaining terms and thereby giving a formula for $\Psi_{\mathrm{Sh},\ell}$ in terms of the Dedekind sums $\D_\ell$.
\item  Applying Lemma~\ref{l:Mtwist}, which relates the values of $\Delta^{(k)}$ on a power series in $\R[[z_1, \dotsc, z_n]]$ to the coefficients appearing in (\ref{e:prdef}).
\end{itemize}

\begin{lemma} \label{l:nomzero} In the evaluation of $\Psi_{\mathrm{Sh},\ell}(A, M, Q, v)$ using $(\ref{e:bern})$, the terms arising from 
tuples $m$ with any component $m_j = 0$ in
$ \Psi(\pi_\ell A \pi_\ell^{-1}, \pi_\ell^{-1}M,  \pi_\ell Q , \pi_\ell v)$ and $\ell \Psi(A, M, Q, v)$ cancel.  In particular, $\Psi_{\mathrm{Sh},\ell}(A, M, Q, v) \in \R[[z_1, \dotsc, z_n]].$
\end{lemma}

\begin{proof}  This is the manifestation of Cassou--Nogu\`es' trick in our context.   Up to the factor $\sgn \det \sigma$,
the value of $\Psi_{\mathrm{Sh}}(A, M, Q, v)$ is  the right side of (\ref{e:bern}) summed over
various cones $C = C(\sigma_{i_1}, \dotsc, \sigma_{i_r})$ and all $a \in \sP \cap (v + \Z^n)$, with the subsets $\{i_j\}$ chosen according the $Q$-perturbation rule, $\sP$  the parallelpiped associated to $C$, and $z$ replaced by $z M^t$:
\begin{equation} \label{e:casum1}
 \sum_{C} \sum_{a \in \sP \cap v + \Z^n} (-1)^r \sum_{\stack{m_j=0}{r\text{-tuples }} }^\infty \prod_{j=1}^{r} 
\frac{B_{m_j}(\sigma^{-1}(a)_{i_j})}{m_j!} (z M^t \sigma_{i_j})^{m_j-1} \prod_{i \not\in \{i_j\}} e^{(z M^t \sigma_i)(\sigma^{-1}(a)_i)}. 
\end{equation}
Let us now fix a cone $C$ and consider the corresponding contribution of $\pi_\ell C$ 
to the value $\Psi_{\mathrm{Sh}}(\pi_\ell A, \pi_\ell^{-1}M,  \pi_\ell Q , \pi_\ell v). $
(Note that $C$ will be included using perturbation via $Q$ if and only if $\pi_\ell C$ will be included using perturbation via $\pi_\ell Q$.)
In applying (\ref{e:bern}), we use the generators $\pi_\ell \sigma_{i_j}$ for the cone $\pi_\ell C$.  
 By applying the change of variables
$a \mapsto \pi_\ell^{-1} a$, we obtain
the exact same expression as (\ref{e:casum1}) except with $\Z^n$ in the second index replaced by $\frac{1}{\ell} \Z \oplus \Z^{n-1}$:
 \begin{equation}
  \sum_{C} \sum_{a \in \sP \cap v + (\frac{1}{\ell}\Z \oplus \Z^{n-1})} (\text{same}).
\label{e:casum2}
 \end{equation}

Fix a tuple $m = (m_1, \dotsc, m_r)$ appearing in the sum (\ref{e:casum1}) such that at least one $m_j$ is equal to zero.
Fix such an index $j$ and a point $a \in \sP \cap v + \Z^n$.   For each equivalence class $b$ mod $\ell$, there is a unique point of the form $a + k\sigma_{i_j}/\ell$ in $\sP \cap v + (\frac{1}{\ell} \Z \oplus \Z^{n-1})$ for an integer $k \equiv b \pmod{\ell}$.
Now, the summand associated to each of these points in (\ref{e:casum2}) is equal to the summand of the associated point $a$ in (\ref{e:casum1}), and in particular is independent of $k$.  To see this, note that \[ \sigma^{-1}(a + k \sigma_{i_j}/\ell) = \sigma^{-1}(a) + 
(0, \dotsc, 0, k/\ell, 0, \dotsc, 0), \]
with $k/\ell$ in the $i_j$th component.  Hence the only term possibly depending on $k$ is $B_{m_j}(\sigma^{-1}(a)_{i_j})$, but $B_0(x) = 1$ is a constant.
The $\ell$ terms $a + k\sigma_{i_j}/\ell$ in (\ref{e:casum2}) therefore cancel with the term $a$ in (\ref{e:casum1}), in view of the factor  $\ell$ in the definition (\ref{e:psildef}).
\end{proof}

\begin{lemma} \label{l:psisleval}
We have
\[
\Psi_{\mathrm{Sh},\ell}(A, M, Q, v) = (-1)^n \sgn\det(\sigma) \sum_{r} \ell^{- \underline{r}} \cdot \D_\ell(\sigma, r+1, Q, v) \frac{ (z M^t \sigma)^{r}}{(r+1)!},
 \]
 where $r$ ranges over all $n$-tuples $r = (r_1, \dotsc, r_n)$ of nonnegative integers.
\end{lemma}

\begin{proof}
We will require the decomposition~(\ref{e:cqdecomp}) for $C_Q \cap (v + \Z^n)$, whose notation we now recall.
 For each $x \in \Z^n/\sigma \Z^n$, let $J = J(x)$ denote the set
of indices $j$ such that $\sigma^{-1}(v + x)_j \in \Z$. For each $I \supset \overline{J}$, consider the cone
$C_I = C(\sigma_i: i \in I)$ with associated parallelpiped $\sP_I$.  The point $x$ and subset $I$ yield a point $a_I \in \sP_I$ such that
$a_I - v \equiv x \pmod{\sigma\Z^n}$, defined by
(\ref{e:aIdef}).

 We evaluate  $\Psi_{\mathrm{Sh}}(A, M, Q, v)$ by employing the decomposition~(\ref{e:cqdecomp}) and applying (\ref{e:bern}).
By Lemma~\ref{l:nomzero}, we need only consider  terms from (\ref{e:bern}) arising from $m_j \ge 1$.   We 
write $r = (r_1, \dotsc, r_n) = (m_1 - 1, \dotsc, m_n - 1)$.
Suppressing for the moment the factor of $\sgn \det(\sigma)$ in the definition (\ref{e:phisdef}) of $\Phi_{\mathrm{Sh}}$, 
we obtain that for a vector of nonnegative integers
$r$ and a class $x \in \Z^n/\sigma \Z^n$, the contribution of the cone $C_{I}$ to the coefficient
of $\prod_{i=1}^{n} (zM^t \sigma_i)^{r_i}$ in $\Psi_{\mathrm{Sh}}(A, M, Q, v)$ for $I \supset \overline{J}$
is 0 unless $r_i = 0$ for $i \not\in I$, and in that case equals
\begin{equation}
\label{e:uini}
  \text{weight}(C_{I}) (-1)^{\#I} \prod_{i \not\in J} \frac{B_{r_i + 1}(\sigma^{-1}(v+x)_i)}{(r_i + 1)!}
\prod_{i \in J \cap I} \frac{b_{r_i+1}(1)}{(r_i+1)!}. 
\end{equation}
Therefore, let $J_r = J \cap \{i: r_i = 0\}$.  The expression (\ref{e:uini}) summed over all $I \supset \overline{J_r}$ can be written
\begin{equation} \label{e:inner}
(-1)^n \sum_{x \in \Z^n/\sigma\Z^n}  \prod_{i \not\in J_r} \frac{B_{r_i + 1}(\sigma^{-1}(v+x)_i)}{(r_i + 1)!} 2^{-\# J_r} \sum_{I \supset \overline{J_r}} \text{weight}(C_{I}) (-2)^{n - \#I}. 
 \end{equation}
The inner sum in (\ref{e:inner}) is easily computed using (\ref{e:weight}):
\[ \sum_{\overline{I} \subset J_r} (-2)^{\#\overline{I}} \prod_{i \in \overline{I}} \frac{1 + \sign(Q\sigma^{-t})_i}{2} = (-1)^{\#J_r}\prod_{j \in J_r}  \sign(Q\sigma^{-t})_j. \]

Therefore, we end up with the following formula for the coefficient of $\prod_{i=1}^{n} (zM^t \sigma_i)^{r_i}$ arising from terms with each $m_i = r_i + 1 \ge 1$:
\begin{equation} \label{e:inner2}
 \sum_{x \in \Z^n/\sigma\Z^n}  \prod_{i \not\in J_r} \frac{B_{r_i + 1}(\sigma^{-1}(v+x)_i)}{(r_i + 1)!} \prod_{j \in J_r} \frac{ -\sgn(Q\sigma^{-t})_j}{2}
 =  \sum_{x \in \Z^n/\sigma\Z^n} \frac{ \B_{r+1}(\sigma^{-1}( x+ v ),\sigma^{-1}Q)}{(r+1)!}.
 \end{equation}
 Evaluating the same expression with $(A, M, Q, v)$ replaced by $(\pi_\ell A \pi_\ell^{-1}, \pi_\ell^{-1} M, \pi_\ell Q, \pi_\ell v)$ and using the definition of $\Psi_{\mathrm{Sh},\ell}$ gives the desired result.
\end{proof}

Theorem~\ref{t:lexplicit} now follows from Lemma~\ref{l:psisleval} and Lemma~\ref{l:Mtwist}  applied with
\[ F =  (-1)^n \sgn\det(\sigma) \sum_{r} \ell^{- \underline{r}} \cdot \D_\ell(\sigma, r+1, Q, v) \frac{ z^{r}}{(r+1)!} \] 
and $M$ replaced by $\sigma^t M$.

\medskip

In \cite{pcsd}, we show that Theorem~\ref{t:lexplicit} implies the following integrality property of $\Psi_{\mathrm{Sh},\ell}$
(see Theorem 4 and \S2.7 of {\em loc.~cit.}).

\begin{theorem} \label{t:smooth} Suppose that $M$ and $v$  satisfy $f_M(v + {\textstyle \frac{1}{\ell}}\Z \oplus \Z^{n-1}) \subset 
\Z[{\textstyle \frac{1}{\ell}}].$ Then for every nonnegative integer $k$, we have $\Delta^{(k)}\Psi_{\mathrm{Sh},\ell}(A, M, Q, v) \in \Z[\frac{1}{\ell}]$.
\end{theorem}

Theorem~\ref{t:smooth} can be used to prove Theorems~\ref{t:integral} and~\ref{t:padic} from the introduction; furthermore Spiess' cohomological formalism for $p$-adic $L$-functions can then be used with our construction to deduce Theorem~\ref{t:oov}.  We refer the reader to \cite[\S3--5]{pcsd} for the proofs.

\subsection{A generalized cocycle} \label{s:gencoc}

We conclude the paper by defining a generalization of the power series $\Psi_{\mathrm{Sh},\ell}$.
 We discuss  this generalization here because it was
stated without proof in \cite[Proposition 2.4]{pcsd}.

Let $\cP = \R[z_1, \dotsc, z_n]$, viewed as a $\Gamma$-module via $(\gamma P)(z) = P(z \gamma).$
The fact that the power series $\Psi_{\mathrm{Sh},\ell}$ is regular (by Lemma~\ref{l:psisleval}) implies that its domain of definition can be expanded from 
matrices $M$ and their associated polynomials $f_M$ to arbitrary polynomials $P \in \cP$.
Let $\tilde{\cF}$ 
denote the $\R$-vector space of functions  $f\colon \cP \times \cQ \times \cV \lra \R$  that
are linear in the first variable and satisfy the  distribution relation
\begin{equation}
 f(P, Q, v) = \sgn(\lambda)^n \sum_{\lambda w = v} f(\lambda^{\deg P} P, \lambda^{-1} Q, w) 
\end{equation}
for each nonzero integer $\lambda$ when $P$ is homogeneous. 
The space $\tilde{\cF}$ has a $\Gamma$-action given by (\ref{e:vaction}), with $M$ replaced by $P$.

Following  (\ref{e:prdef}), define for $P \in \cP$ and any matrix $\sigma \in M_n(\R)$ coefficients $P_r(\sigma)$ by 
\begin{equation}
P(z\sigma^t)=\sum_r \frac{P_r(\sigma)}{r!} z_1^{r_1}\cdots z_n^{r_n}.
\end{equation}
Fixing $\sigma = 1$, these coefficients define an operator
$\Delta^{(P)}\colon \R[[z_1, \dotsc, z_n]] \rightarrow \R$ given by
\[ \Delta^{(P)}\left( \sum_r F_r z^r \right) = \sum_r F_r P_r(1). \]
\begin{prop} The function
\begin{align}
 \tilde{\Psi}_{\mathrm{Sh},\ell}(A, P, Q, v) :=& \ \Delta^{(P)}\Psi_{\mathrm{Sh},\ell}(A, 1, Q, v)  \label{e:psitdef} \\
 =& \ (-1)^n \sgn( \det \sigma)\sum_{r} \frac{P_r(\sigma)}{\ell^{\underline{r}}(r+1)!} \D_\ell(\sigma, r + 1, Q, v)
\label{e:pcsdform}  \end{align}
is a homogeneous cocycle for $\Gamma_\ell$ valued in $\tilde{\cF}$, i.e.\ $\tilde{\Psi}_{\mathrm{Sh},\ell} \in Z^{n-1}(\Gamma_\ell, \tilde{\cF}).$
\end{prop} 
\begin{proof}
The cocycle condition 
 $ \sum_{i=0}^{n} (-1)^i \tilde{\Psi}_{\mathrm{Sh},\ell}(A_0, \dotsc, \hat{A}_i, \dotsc, A_n) = 0 $
 follows from that for $\Psi_{\mathrm{Sh},\ell}$.  The fact that
$\tilde{\Psi}_{\mathrm{Sh},\ell}$ is invariant under $\Gamma_\ell$ follows from the equivalent statement
for $\Psi_{\mathrm{Sh},\ell}$ and the  fact that for any matrix $\gamma \in M_n(\R)$ and any 
$F \in \R[[z_1, \dots, z_n]]$, we have
\begin{equation} \label{e:gammatwist}
\Delta^{(\gamma^t P)} F(z) = \Delta^{(P)} F(z\gamma).
\end{equation}
Equation (\ref{e:gammatwist}) is a mild generalization of Lemma \ref{l:Mtwist} that again follows from
(\ref{e:crs}).
The equality between (\ref{e:psitdef}) and (\ref{e:pcsdform}) follows from Lemma \ref{l:psisleval} and (\ref{e:gammatwist})
applied to $\gamma = \sigma$.
\end{proof}

 \bigskip
 \bigskip

 P.~C. :  \textsc{Institut de Math\'ematiques de Jussieu, Universit\'e Paris 6,  France} 

 \textit{E-mail} \textbf{charollois (at) math (dot) jussieu (dot) fr}

 \medskip

 S.~D. : \textsc{Dept. of Mathematics, University of California  Santa Cruz, USA} 

 \textit{E-mail} \textbf{sdasgup2 (at) ucsc (dot) edu}
 
 \medskip
 
 M.~G. : \textsc{Dept. of Mathematics \& Statistics, University of Calgary, Canada} 

 \textit{E-mail} \textbf{mgreenbe (at) ucalgary (dot) ca}

\end{document}